\documentclass[draft,12pt,a4paper]{article}
\usepackage{a4wide}
\usepackage{amssymb}
\usepackage{amsthm}
\usepackage[english]{babel}
\usepackage[utf8]{inputenc}
\usepackage{amsmath}
\usepackage{longtable}
\usepackage{url,verbatim}
\usepackage{multicol}

\newtheorem{thm}{Theorem}
\newtheorem{corollary}{Corollary}

\newtheorem{proposition}{Proposition}

\newtheorem{lemma}{Lemma}

\begin{document}

\title{Enumeration of Hypermaps of a Given Genus\footnote{The authors wish to
thank Alexander Mednykh, Roman Nedela and the referees for helpful suggestions
to improve the presentation of this article.}}

\author{ Alain Giorgetti$^1$\footnote{For this work Alain Giorgetti was supported
by the French ``Investissements d’Avenir'' program, project ISITE-BFC (contract ANR-15-IDEX-03).} $\quad$ Timothy~R.~S. Walsh$^2$\\
\\
$^1$ FEMTO-ST institute, Univ. Bourgogne Franche-Comt\'e, CNRS\\
 16 route de Gray, 25030 Besan\c{c}on cedex, France\\
\texttt{alain.giorgetti@femto-st.fr} \\
\phantom{vertical space}\\
$^2$ Department of Computer Science\\
 University of Quebec in Montreal (UQAM)\\
P.O.~Box~8888, Station A, Montreal, Quebec, Canada, HC3-3P8\\
\texttt{walsh.timothy@uqam.ca} 
}

\date{
\textit{
 ARS MATHEMATICA CONTEMPORANEA 15 (2018)}
 \url{https://amc-journal.eu/index.php/amc/article/view/1115/1221}}

\maketitle


\begin{abstract}
This paper addresses the enumeration of rooted and unrooted hypermaps of a given
genus. For rooted hypermaps the enumeration method consists of considering the
more general family of multirooted hypermaps, in which darts other than the root
dart are distinguished. We give functional equations for the generating series
counting multirooted hypermaps of a given genus by number of darts, vertices,
edges, faces and the degrees of the vertices containing the distinguished darts.
We solve these equations to get parametric expressions of the generating
functions of rooted hypermaps of low genus. We also count unrooted hypermaps of
given genus by number of darts, vertices, hyperedges and faces.
\end{abstract}




\noindent\textit{Math. Subj. Class.: 05C30, 05A15} 


\section{Introduction}
\label{intro:sec}
A \emph{(combinatorial) hypermap} is a triple $(D,R,L)$ where $D$ is a finite
set of \emph{darts}
 and $R$ and $L$ are permutations on $D$ such that the group $\langle
 R,L\rangle$ generated by $R$ and $L$ acts transitively on $D$.
 A \emph{(combinatorial ordinary) map} is a hypermap
$(D,R,L)$ whose permutation $L$ is a fixed-point-free involution on $D$. For a
hypermap (resp. map) the orbits of $R$, $L$ and $RL$ ($L$ followed by
$R$) are respectively called \emph{vertices}, \emph{hyperedges} (resp.
\emph{edges}) and \emph{faces}. The \emph{degree} of a vertex, edge, hyperedge
or face is the number of darts it contains. The equivalence of combinatorial maps and
topological maps having been established in~\cite{JoSi}, we use the word ``map'' to mean ``combinatorial
map'' throughout this paper.
 The \emph{genus} $g$ of a map is given by the Euler-Poincaré
 formula~\cite{Cox73}
\begin{equation}
v - e + f = 2(1 – g),  \label{euler:eq}
\end{equation} where $v$ is the number of vertices, $e$ is the number of edges
and $f$ is the number of faces.  The genus of a hypermap with $t$
darts, $v$ vertices, $e$ hyperedges and $f$ faces was defined in~\cite{Jac68} by the
formula
\begin{equation}
v + e + f = t + 2(1 – g). \label{genus:formula:hyper}
\end{equation}

An \emph{isomorphism} between two maps or hypermaps $(D, R, L)$ and $(D', R',
L')$ is a bijection from $D$ onto $D'$ that takes $R$ into $R'$ and $L$ into $L'$; it
corresponds to an orientation-preserving homeomorphism between two topological
maps.  A \emph{sensed} hypermap (resp. map) is an isomorphism class of hypermaps
(resp. maps). We admit the existence of a unique hypermap (resp. map) with an
empty set of darts $D$, called the \emph{empty} hypermap (resp. map). For both
of these objects $v = f = 1$ and $g = e = 0$. A \emph{rooted} hypermap (resp.
map) is either the empty hypermap (resp. map) or a tuple $(D,x,R,L)$ where
$(D,R,L)$ is a non-empty combinatorial hypermap (resp.
map) and $x \in D$ is a distinguished dart, called the \emph{root}.


The enumeration of maps and hypermaps has several non-trivial applications.  One
such application is based on the correspondence between hypermaps and algebraic
curves established by the Belyi theorem~\cite{LZ04}. For instance, the formula
for the number of plane trees was used by A. Zvonkin in the computer generation
of Shabat polynomials of bounded degree~\cite{LZ04}.  Another area where the map
enumeration plays an important role is theoretical physics, in particular in
2-dimensional gravitation models.  Roughly speaking, map enumeration is used to
compute matrix integrals determining the properties of gravitational fields (see
for instance the works of B. Eynard~\cite{Eynard2011}). Some hypermaps have been
shown to be related to contextuality in quantum physics~\cite{pghs15:ij}.
Also, A. Mednykh and R. Nedela have applied the enumeration of rooted (resp.
unrooted) hypermaps to the enumeration of subgroups (resp. conjugacy classes of
subgroups) of the triangle group with three generators $x$, $y$, $z$ and the
relation $xyz=1$~\cite{Mednykh2017}.

We enumerate rooted hypermaps of a given
genus by number of darts, vertices, hyperedges and faces.
To do so we consider more general families of rooted hypermaps and bipartite
maps, in which other vertices or darts than the root dart are distinguished.
We also use the genus-preserving bijection between hypermaps and
2-vertex-coloured bipartite maps presented in~\cite{Walsh75}. But since
bipartite maps have all their \textbf{faces} of even degree and we're using the
degrees of the \textbf{vertices} as parameters, we must instead study the
face-vertex dual of a 2-coloured bipartite map, that is, a map whose faces are
coloured in two colours (white and black) so that no two faces that share an
edge have the same colour.
All these maps are \textit{Eulerian} -- that is, all their vertices are of even
degree -- but not all Eulerian maps are 2-face-colourable. For example, the map on the
torus with one vertex, one face and two edges is Eulerian because its only
vertex is of degree 4, but its face cannot be coloured because it shares both
edges with itself. Therefore we call the maps we are studying
\emph{face-bipartite}.

A \emph{sequenced (rooted) map} is a rooted map with some vertices other than
the \emph{root vertex} (the vertex that contains the root) distinguished from
each other and from all the other vertices. The labels that distinguish these
vertices can be taken to be $1$, $2$, \ldots $k$, where $k$ is the number of
distinguished vertices. A \emph{sequenced (rooted) hypermap} is defined
similarly. We state (in Section~\ref{seq:hypermap:sec}) a bijective
decomposition for the set $\mathcal{H}(g,t,f,e,n,D)$ of sequenced orientable
hypermaps of genus $g$ with $t$ darts, $f$ faces and $e$ hyperedges, with the
root vertex of degree $n$ and with the sequence of degrees of the distinguished
vertices equal to $D = (d_1, d_2, \ldots d_{|D|})$, where $d_i$ is the degree of
the distinguished vertex with label $i$.
We obtain a bijective decomposition of the set $\mathcal{F}(g,e,w,b,n,D)$ of
sequenced orientable face-bipartite maps of genus $g$ with $e$ edges, $w$ white faces, $b$
black faces, with the root face of degree $2n$ and with the sequence of
half-degrees of the distinguished vertices equal to $D$. Then we apply
face-vertex duality to obtain a bijective decomposition of the corresponding set of
2-coloured bipartite maps with distinguished faces. Next we use the bijection
in~\cite{Walsh75} to obtain a bijective decomposition for hypermaps with distinguished faces,
and finally we again apply face-vertex duality to obtain a bijective decomposition of
$\mathcal{H}(g,t,f,e,n,D)$.

A \emph{mutirooted hypermap} is a hypermap in which a non-empty sequence of
 darts with pairwise distinct initial vertices is distinguished. We relate
multirooted hypermaps to sequenced
 hypermaps and thus obtain a recurrence for the number of
multirooted hypermaps and functional equations for the generating series
counting multirooted hypermaps of a given genus by number of darts, vertices,
edges, faces and the degrees of the initial vertices of the distinguished darts.

The paper is organized as follows. Section~\ref{background:sec} fixes some
notations, 
recalls a known decomposition for sequenced rooted maps and describes the
bijection between hypermaps and bipartite maps presented
in~\cite{Walsh75}. 
Sections~\ref{eulerian:sec} and~\ref{seq:hypermap:sec} respectively
enumerate sequenced face-bipartite maps and sequenced rooted hypermaps
of a given genus.
In Section~\ref{multi:hypermap:sec} we consider multirooted hypermaps and we
give equations for the generating functions that count these objects. In
Section~\ref{rooted:hyper:series:sec} we give functional equations relating the
generating functions for rooted hypermaps with that for multirooted hypermaps.
Then we show how to solve these equations.
In Section~\ref{hyp:series:explicit:small:genus:sec} we obtain parametric
expressions for the generating functions that count rooted hypermaps with a
given small positive genus.
Section~\ref{unrooted:hyp:sec} presents enumeration algorithms for sensed
unrooted hypermaps counted by number of darts, vertices and hyperedges.
Appendix \ref{rooted:appendix} (resp. \ref{unrooted:appendix}) contains a table
for numbers of rooted (resp. unrooted) hypermaps of genus $g$ with $d$ darts, $v$
vertices and $e$ hyperedges for $d \leq 14$.

\section{Background}
\label{background:sec}

\subsection{Notations}
We first introduce the notations and conventions we use throughout the
paper. Let $D$ and $D'$ be two lists of integers.
The inclusion $D' \subseteq D$ means that $D'$ is a sublist of $D$.
In this case $D - D'$ is the complementary sublist of $D'$ in $D$. For instance,
the sublists of $D = [1,1,2]$ are the empty list $[\,]$, $[1]$ (twice), $[2]$,
$[1,1]$, $[1,2]$ (twice) and $D$ itself. Their complementary sublists in the
same order are $D$, $[1,2]$ (twice), $[1,1]$,$[2]$, $[1]$ (twice) and $[\,]$. We
denote by $D.D'$ the concatenation of the lists $D$ and $D'$. If $i$ is an
integer and $D$ is a list of integers, then $i.D$ is a shortcut for $[i].D$. For
$1 \leq j \leq |D|$ we denote by $d_j$ the $j$-th element of the list $D$ of
length $|D|$ and by $D - \{d_j\}$ the list obtained from $D$ by removing its
$j$-th element $d_j$. Let $\rho$ be a positive integer. The abbreviation
$D_{1..\rho}$ denotes the list $[d_1, \ldots, d_{\rho}]$. The abbreviation
$v_{1..\rho}^{D_{1..\rho}}$ denotes $v_{1}^{d_{1}} \ldots v_{\rho}^{d_{\rho}}$.

The sign $+$ (resp. $\sum$) denotes (resp. generalized) disjoint set union in
the following decompositions and (resp. generalized) arithmetic sum in the
following equations. By convention, a disjoint set union (resp. sum) over an
empty domain is equal to the empty set (resp. zero).
For any logical formula $\varphi$ the notation $\Delta_{\varphi}$ means the
singleton set containing only the empty hypermap or map (depending on the
context) and the empty set if $\varphi$ is false. The notation
$\delta_{\varphi}$ means $1$ if $\varphi$ is true and $0$ if $\varphi$ is false.

\subsection{Bijective decomposition of the set of sequenced maps}
\label{erratum:sec}

In 1962 W. T. Tutte~\cite{Tut62} presented a bijective decomposition of a planar
map with all the vertices distinguished and a root in every vertex. In 1972 T.
R. Walsh and A. B. Lehman~\cite{W1} generalized this decomposition to maps of
higher genus and used it to count rooted maps of a given genus by number of
vertices and faces. 
In 1987 D. Arquès~\cite{Arq87a} used this latter decomposition to find a
closed-form formula for the number of rooted maps of genus 1 by number of
vertices and faces.
In 1991 E. A. Bender and E. A. Canfield~\cite{BC91}
presented a more efficient decomposition that roots only a single vertex and
distinguishes only as many other vertices as necessary and used it to obtain
explicit formulas for counting rooted maps of genus 2 and 3. In 1998 the first
author~\cite{gio98b} modified this decomposition and used it to obtain a
bijective decomposition satisfied by the set $\mathcal{M}(g,e,f,n,D)$ of
sequenced orientable maps of genus $g$ with $e$ edges and $f$ faces, with the
root vertex of degree $n$ and with $D$ the list of degrees of the
distinguished vertices was obtained in~\cite{gio98b}. Since this bijective
decomposition contains an error, we present the correct bijective decomposition
here, and we derive it to make the derivation more accessible than the contents
of a Ph. D. thesis.

\begin{thm}
The set $\mathcal{M}(g,e,f,n,D)$ of sequenced orientable maps of
genus $g$ with $e$ edges and $f$ faces, with the root vertex of degree $n$ and
with the list $D$ of degrees of the distinguished vertices is defined by the
bijective decomposition
\begin{eqnarray}
& & \mathcal{M}(g,e,f,n,D) = \nonumber \\ 
 & & \quad \sum_{
    \begin{scriptsize}
    \begin{array}{c} 
     g_1 + g_2 = g \\
     e_1 + e_2 = e - 1 \\
     f_1 + f_2 = f\\
     n_1+n_2 = n - 2 \\
     D_1 \subseteq D 
     \end{array}
    \end{scriptsize}
   } 
   \mathcal{M}(g_1,e_1,f_1,n_1,D_1) \times \mathcal{M}(g_2,e_2,f_2,n_2,D-D_1) \nonumber  \\
 & & \quad + 
 \sum_{p=1}^{n-3} \mathcal{M}(g-1,e-1,f,n-2-p,p.D)
 \times \{1,\ldots,p\} \label{correction:bij:orientable:edge} \\
 & & \quad + \sum_{p = n -1}^{p = 2e -2} \mathcal{M}(g,e-1,f,p,D) \nonumber  \\
 & & \quad + \sum_{j=1}^{|D|} 
 \mathcal{M}(g,e-1,f,d_j+n-2,D - \{d_j\}) + \Delta_{(g,e,f,n,D)=(0,0,1,0,[\,])}. \nonumber
\end{eqnarray}
\end{thm}

\begin{proof}
If a map $m$ has at least one edge, we reduce by $1$ the number of edges by the
face-vertex dual of deleting the root edge. There are two cases of this
operation, depending upon whether the root edge is a loop or a link, and each of
these cases breaks down into two sub-cases.

\noindent \textbf{Case 1: The root edge is a loop.} We delete the root edge and split
the root vertex into two parts, $s_1$ and $s_2$. If $r$ is the root, then $s_1$
consists of the darts $R(r)$, $R^2(r)$, \ldots, $R^{-1}(L(r))$ and $s_2$
consists of the darts $R(L(r))$, $R^2(L(r))$, \ldots, $R^{-1}(r)$. This case
breaks down into two cases, depending upon whether or not this operation
disconnects the map.

\textbf{Case 1a: This operation disconnects the map into two maps}, $m_1$
containing $s_1$ and $m_2$ containing $s_2$. If $m_1$ has at least $1$ edge, its
root is $r_1 = R(r)$, and if $m_2$ has at least $1$ edge, its root is $r_2 =
R(L(r))$. Let $g_1$, $e_1$, $f_1$, $n_1$, $D_1$ and $g_2$, $e_2$, $f_2$, $n_2$,
$D_2$ be the parameters of the maps $m_1$ and $m_2$, respectively, corresponding
to $g$, $e$, $f$, $n$, $D$.
This operation reduces by $1$ the total number of edges; so $e_1 + e_2 = e – 1$.
It leaves unchanged the total number of faces because $r$ and $L(r)$ simply get
deleted from the cycle(s) of $RL$ ($L$ followed by $R$) containing them; so $f_1
+ f_2 = f$. It increases by $1$ the total number of vertices; so from Formula
(\ref{euler:eq}), which relates the genus of a map
to the number of its vertices, faces and edges, it can easily be deduced that
$g_1 + g_2 = g$. It decreases by $2$ the total number of darts in $s_1$ and $s_2$ since $r$ and
$L(r)$, which belonged to the root vertex, get eliminated; so $n_1 + n_2 = n –
2$. Finally, $D_1$ can be any sublist of $D$ and $D_2$ is just the
complementary sublist, denoted by $D - D_1$. This operation is uniquely
reversible; so the set of ordered pairs of sequenced maps obtained in this case is 

\begin{equation}
\sum_{
    \begin{scriptsize}
    \begin{array}{c} 
     g_1 + g_2 = g \\
     e_1 + e_2 = e - 1 \\
     f_1 + f_2 = f \\
     n_1+n_2 = n - 2 \\
     D_1 \subseteq D 
    \end{array}
    \end{scriptsize}
   } 
   \mathcal{M}(g_1,e_1,f_1,n_1,D_1) \times  \mathcal{M}(g_2,e_2,f_2,n_2,D-D_1),
\label{case1a:map:term}
\end{equation}
   
\noindent where $\Sigma$ means the union of disjoint sets.

\textbf{Case 1b: This operation does not disconnect the map}, but instead turns
it into a new map $m'$ with $e – 1$ edges and $f$ faces and, since the number of
vertices increases by $1$, the genus of  $m'$ is $g - 1$, so that this case only
occurs when $g \geq 1$. Neither $s_1$ nor $s_2$ can be of degree $0$ (otherwise
the map would be disconnected); so we can choose for $m'$ the root $r_1 = R(r)$
belonging to $s_1$. Let $p$ be the degree of $s_2$. Since the sum of the degrees
of $s_1$ and $s_2$ is $n – 2$, the degree of $s_1$, the root vertex, is $n – 2 –
p$. We distinguish the vertex $s_2$ so that this operation can be   reversed,
and we put its degree $p$ at the beginning of the list $D$, turning it into
$p.D$. Now this operation is reversible in $p$ distinct ways, since any of the
$p$ darts of $s_2$ can be chosen to be $R(L(r))$ when we merge the vertices
$s_1$ and $s_2$ and replace the deleted root edge. Now $p$ can be any integer
from $1$ up to $n – 3$ (so that $n – 2 – p \geq 1$). For both $p$ and $n –
2 – p$ to be at least $1$, $n$ must be at least $4$. The set of sequenced maps
obtained in this case is

\begin{equation}
\sum_{p=1}^{n-3} \mathcal{M}(g-1,e-1,f,n-2-p,p.D) \times \{1,\ldots,p\}.
 \label{case1b:map:term}
\end{equation}

\noindent \textbf{Case 2: The root edge is a link.} We contract the root edge,
merging its two incident vertices $s_1$  containing the root $r$ and $s_2$
containing $L(r)$ into a single vertex $s$ with root $R(r)$. This operation
decreases by $1$ the number of edges and doesn't change the number of faces,
since $r$ and $L(r)$ simply get deleted from the cycle(s) containing them. Since
the number of vertices is decreased by $1$, the genus remains the same. This
case breaks down into two sub-cases, depending upon whether or not $s_2$ is one
of the distinguished vertices.

\subparagraph{Case 2a: The vertex $s_2$ is not one of the distinguished
vertices.} Let $p$ be the degree of the new vertex $s$. Then $p = n –2 +$ the
degree of $s_2$, and since the degree of $s_2$ must be at least $1$,  we have $p
\geq  n – 1$. Also, the new map has $2e – 2$ darts; so $p \leq  2e – 2$. This
operation is uniquely reversible for each value of $p$; so the set of maps so
obtained is

\begin{equation}
\sum_{p = n -1}^{p = 2e-2} \mathcal{M}(g,e-1,f,p,D).
\label{case2a:map:term}
\end{equation}

\subparagraph{Case 2b: The vertex $s_2$ is one of the distinguished vertices.}
It can be any one of the $|D|$ distinguished vertices. If it is the $j$th
distinguished vertex, then its degree is $d_j$. Then since it  gets merged with
$s_1$ into the new root vertex, $d_j$ gets dropped from $D$. Finally, the degree
of $s$ is $d_j + n –2$. This operation too is uniquely reversible; so the set of
maps so obtained is

\begin{equation}
 \sum_{j=1}^{|D|} 
 \mathcal{M}(g,e-1,f,d_j+n-2,D - \{d_j\}).  
\label{case2b:map:term}
\end{equation}

Finally, suppose that $m$ has no edges. It is of genus $0$, has $1$ face,
its one vertex is of degree $0$ and its list $D$ is empty because it has no
distinguished vertices; so it constitutes the  singleton 
\begin{equation}
\Delta_{(g,e,f,n,D)=(0,0,1,0,[\,])}. 
\label{case3:map:term}
\end{equation}

Then $\mathcal{M}(g,e,f,n,D)$ is the disjoint union of the sets given by
(\ref{case1a:map:term})-(\ref{case3:map:term}).
\end{proof}

\subsection{Bipartite maps and hypermaps}
\label{bipartite:hyper:sec}
To motivate the transformation of (\ref{case1a:map:term})-(\ref{case3:map:term})
into the corresponding equations for sequenced hypermaps we briefly describe the
bijection in~\cite{Walsh75} that takes a hypermap $h$ into a 2-coloured
bipartite map $m = I(h)$, its \textit{incidence map}. The bijection $I$ takes
the darts, vertices and hyperedges of $h$ into the edges, white vertices and black vertices
of $m$.
A root (distinguished dart) of $h$ corresponds to a distinguished \textbf{edge}
of $m$; to make it correspond to a root of $m$ we impose the condition that a
root of $m$ belongs to a white vertex. The permutation $R$ in $h$ corresponds to
$R$ in $m$ acting on a dart in a white vertex and the permutation $L$ in $h$
corresponds to $R$ in $m$ acting on a dart in a black vertex. The permutation
$L$ in $m$ doesn't correspond to any permutation in $h$; rather, since it takes
a dart belonging to a vertex of one colour into a dart belonging to a vertex of
the opposite colour, it toggles $R$ in $m$ between $R$ and $L$ in $h$. A face
(cycle of $RL$) in $h$ corresponds to a face in $m$ with twice the degree. To
see this, we follow one application of $RL$ in $h$ starting with a dart $d$,
which corresponds to an edge in $m$ but we make it correspond to the dart $d'$
in that edge that also belongs to a white vertex. Then the $L$ in $h$ takes $d'$
first into $L(d')$, which belongs to a black vertex, and then into $RL(d')$ and
the following $R$ in $h$ takes $RL(d')$ first into $LRL(d')$, which belongs to a
white vertex, and then into $RLRL(d')$. Since the genus of a hypermap with $t$
darts, $v$ vertices, $e$ hyperedges and $f$ faces is defined
by~(\ref{genus:formula:hyper}), $m$ has the same genus as $h$.

Since the root of an incidence map of a rooted hypermap must belong to a white
vertex, we impose the condition on a rooted 2-face-coloured face-bipartite map
that the root belong to a white face and we transform
(\ref{case1a:map:term})-(\ref{case3:map:term}) into the corresponding bijective
decomposition for these maps.

\section{Sequenced face-bipartite maps}
\label{eulerian:sec}
Let $\mathcal{F}(g,e,w,b,n,D)$ be the set of sequenced orientable face-bipartite
maps of genus $g$ with $e$ edges, $w$ white faces, $b$ black faces, with the
root face of degree $2n$ and with the list of half-degrees of the
distinguished vertices equal to $D$. For any dart $d$ we denote by $f(d)$ the
face containing $d$ and we note that the face $f(R(d)) = f(L(d))$ must have the
opposite colour from $f(d)$ because those two faces share the edge $\{d,
L(d)\}$.

\begin{thm}\label{face:bipartite:th}
The set $\mathcal{F}(g,e,w,b,n,D)$  satisfies the bijective
decomposition

\begin{eqnarray}
& &\mathcal{F}(g,e,w,b,n,D) = \nonumber \\ 
 & & \quad \sum_{
    \begin{scriptsize}
    \begin{array}{c}
     g_1 + g_2 = g \\
     e_1 + 	e_2 = e - 1 \\
     w_1 + b_2 = b \\
     w_2 + b_1 = w \\
     n_1+n_2 = n - 1 \\
     D_1 \subseteq D 
    \end{array}
    \end{scriptsize}
   } 
   \mathcal{F}(g_1,e_1,w_1,b_1,n_1,D_1) \times \mathcal{F}(g_2,e_2,w_2,b_2,n_2,D-D_1) \nonumber 
   \\
 & & \quad + 
          \sum_{p=1}^{n-2} \mathcal{F}(g-1,e-1,b,w,n-1-p,p.D)
 \times \{1,\ldots,p\} \label{bij:eulerian} \\
 & & \quad + \sum_{p = n}^{p = e-1} \mathcal{F}(g,e-1,b,w,p,D) \nonumber  \\
 & & \quad + \sum_{j=1}^{|D|} 
 \mathcal{F}(g,e-1,b,w,d_j+ n-1,D - \{d_j\}) + \Delta_{(g,e,w,b,n,D)=(0,0,1,0,0,[\,])}. \nonumber 
\end{eqnarray}
\end{thm}

\begin{proof}
\noindent \textbf{Case 1: The root edge is a loop.} By definition, $f(r)$, where $r$ is
the root of the map $m$, is white, so that since $r_1 = R(r)$, $f(r_1)$ must be
black. But when the loop is removed and the vertex $s$ containing $r$ is split,
$r_1$ becomes a root; so $f(r_1)$ must change colour and so must all the faces
of the new map $m'$ (in case 1b) or the map $m_1$ containing $r_1$ (in case 1a). 
 In case 1a, the other map $m_2$ has $r_2 = RL(r)$ as a root and $f(r_2)$ is
white; so its faces stay the same colour. This implies that in case 1a $w_1 +
b_2 = b$ and $w_2 + b_1 = w$, whereas in case 1b $w$ and $b$ switch in going
from $m$ to $m'$.
    
In case 1a, we have, as for general maps, $g_1 + g_2 = g$, $e_1 + e_2 = e – 1$
and $D_1$ is any subset of $D$, but instead of $n_1 + n_2 = n - 2$ we have $n_1
+ n_2 = n -1$ because the degrees satisfy the equation $ 2n_1 + 2n_2 = 2n - 2$.
  The analogue of formula (\ref{case1a:map:term}) is thus

\begin{equation}
\sum_{
    \begin{scriptsize}
    \begin{array}{c}
     g_1 + g_2 = g \\
     e_1 + e_2 = e - 1 \\
     w_1 + b_2 = b \\
     w_2 + b_1 = w \\
     n_1+n_2 = n - 1 \\
     D_1 \subseteq D 
    \end{array}
    \end{scriptsize}
   } 
   \mathcal{F}(g_1,e_1,w_1,b_1,n_1,D_1) \times \mathcal{F}(g_2,e_2,w_2,b_2,n_2,D-D_1).
\label{case1a:face:bipartite:term}
\end{equation}

In case 1b, the reduced map $m'$ is still of genus $g – 1$ and has $e – 1$
edges, but the degree of $s_2$ is now $2p$ instead of $p$ and the degree of the
new root vertex $s_1$ is $2(n – 1 – p)$; so the parameter $n – 2 – p$ in
(\ref{case1b:map:term}) changes to $n – 1 – p$. Also, $1 \leq  2p \leq  2n-3$,
but since $2p$ is even, we have $1 \leq  p \leq  n – 2$ instead of $1 \leq  p
\leq n –3$, and the condition that $n \geq  4$ changes to $n \geq  3$. The
analogue of formula (\ref{case1b:map:term}) is thus

\begin{equation}
         \sum_{p=1}^{n-2} \mathcal{F}(g-1,e-1,b,w,n-1-p,p.D)
 \times \{1,\ldots,p\}.
\label{case1b:face:bipartite:term}
\end{equation}

\noindent \textbf{Case 2: The root edge is a link.} Since the new root $R(r)$
belongs to a black face, all the faces change colour; so $b$ and $w$ switch.

In case 2a, we have $2n – 1 \leq  2p \leq  2e – 2$, but since $2p$ is even, we
now have $n \leq  p \leq  e – 1$; so the analogue of (\ref{case2a:map:term})
is 

\begin{equation}
\sum_{p = n}^{p = e-1} \mathcal{F}(g,e-1,b,w,p,D).
\label{case2a:face:bipartite:term}
\end{equation}

In case 2b, the degree of the new root vertex is $2d_j + 2n - 2$; so the
analogue of (\ref{case2b:map:term}) is

\begin{equation}
\sum_{j=1}^{|D|}  \mathcal{F}(g,e-1,b,w,d_j+ n-1,D - \{d_j\}).
\label{case2b:face:bipartite:term}
\end{equation}

Finally, the map with no edges has one white face and no black ones; so the
analogue of (\ref{case3:map:term}) is

\begin{equation}
\Delta_{(g,e,w,b,n,D)=(0,0,1,0,0,[\,])}.
\label{case3:face:bipartite:term}
\end{equation}
\end{proof}

After deriving this bijective decomposition, we became aware of the
article~\cite{DOPS14v1}, which presents a similar bijective decomposition but
for multi-rooted face-bipartite maps, which are like sequenced face-bipartite
maps except that every distinguished vertex has a root.  However, we present our
derivation here for several reasons:
it makes our article self-contained, we obtained it independently
of~\cite{DOPS14v1} and our main purpose is to count hypermaps rather than
face-bipartite maps.  Now~\cite{DOPS14v1} does present a construction that
converts a hypermap into a face-bipartite map.  However, that construction is
not proved and it is far more complicated than the one in~\cite{Walsh75}, which
is not cited in~\cite{DOPS14v1}. We also recently became aware of the
article~\cite{CF15}, which generalizes the results of~\cite{KZ15} by computing
the generating functions for edge-labelled bipartite maps on an orientable
surface of genus $g$ with an unbounded number of faces and including the degrees
of these faces as parameters.

\section{Sequenced rooted hypermaps}
\label{seq:hypermap:sec}

Theorem~\ref{face:bipartite:th} holds for rooted 2-coloured bipartite maps with
distinguished faces, where $e$ is the number of edges, $w$ is the number of
white vertices, $b$ is the number of black vertices, $n$ is half the degree of
the root face and $D$ is the list of half-degrees of the distinguished
faces.  By the bijection described in Section~\ref{bipartite:hyper:sec}, it also
holds for rooted hypermaps with distinguished faces, where $e$ is the number of
darts, $w$ is the number of vertices, $b$ is the number of hyperedges, $n$ is
the degree of the root face and $D$ is the list of degrees of the
distinguished faces. By duality, the theorem also holds for sequenced hypermaps,
where $e$ is the number of darts, $w$ is the number of faces, $b$ is the number
of hyperedges, $n$ is the degree of the root vertex and $D$ is the list of
degrees of the distinguished vertices. To make the letters correspond to the
objects they represent, we change $\mathcal{F}$ to $\mathcal{H}$, $e$ to $t$,
$w$ to $f$ and $b$ to $e$. We thus obtain the following results.

\begin{thm}[Bijective decomposition for sequenced hypermaps]
\label{seq:hyper:th}
Let $\mathcal{H}(g,t,f,e,n,D)$  be the set of sequenced orientable hypermaps of
genus $g$ with $t$ darts, $f$ faces and $e$ hyperedges, with the root vertex of
degree $n$ and with the list of degrees of the distinguished vertices equal
to $D = (d_1, d_2, \ldots d_{|D|})$, where $d_i$ is the degree of the
distinguished vertex with label $i$. The set $\mathcal{H}(g,t,f,e,n,D)$
satisfies the bijective decomposition

\begin{eqnarray}
& & \mathcal{H}(g,t,f,e,n,D) = \nonumber \\
 & & \quad \sum_{
    \begin{scriptsize}
    \begin{array}{c}
     g_1 + g_2 = g \\
     t_1 + t_2 = t - 1 \\
     f_1 + e_2 = e \\
     f_2 + e_1 = f \\
     n_1+n_2 = n - 1 \\
     D_1 \subseteq D 
    \end{array}
    \end{scriptsize}
   } 
   \mathcal{H}(g_1,t_1,f_1,e_1,n_1,D_1) \times \mathcal{H}(g_2,t_2,f_2,e_2,n_2,D-D_1) \nonumber 
   \\
 & & \quad + 
         \sum_{p=1}^{n-2} \mathcal{H}(g-1,t-1,e,f,n-1-p,p.D)
 \times \{1,\ldots,p\} \label{bij:seq:hypermap} \\
 & & \quad + \sum_{p = n}^{p = t-1} \mathcal{H}(g,t-1,e,f,p,D) \nonumber  \\
 & & \quad + \sum_{j=1}^{|D|} \mathcal{H}(g,t-1,e,f,d_j+ n-1,D - \{d_j\}) 
  + \Delta_{(g,t,f,e,n,D)=(0,0,1,0,0,[\,])}. \nonumber 
\end{eqnarray}
\end{thm}

\begin{corollary}[Recurrence between numbers of sequenced hypermaps]
\label{hyp:sec:rec:number:cor}
Let $H(g,t,f,e,n,D)$ be the number of rooted sequenced hypermaps of genus $g$
with $t$ darts, $f$ faces and $e$ hyperedges such that the root vertex is of
degree $n$ and $D$ is the list of degrees of the distinguished vertices.
Then $H(0,0,1,0,0, [\,]) = 1$ and if $t \geq 1$, then 

\begin{eqnarray}
& & H(g,t,f,e,n,D) = \nonumber \\
 & & \quad \sum_{
    \begin{scriptsize}
    \begin{array}{c}
     g_1 + g_2 = g \\
     t_1 + t_2 = t - 1 \\
     f_1 + e_2 = e \\
     f_2 + e_1 = f \\
     n_1+n_2 = n - 1 \\
     D_1 \subseteq D 
    \end{array}
    \end{scriptsize}
   } 
   H(g_1,t_1,f_1,e_1,n_1,D_1) \, H(g_2,t_2,f_2,e_2,n_2,D-D_1) \nonumber 
   \\
 & & \quad + \ \delta_{n \geq 3} \delta_{g \geq 1} 
         \sum_{p=1}^{n-2} p \, H(g-1,t-1,e,f,n-1-p,p.D) \label{rec:seq:hypermap}
         \\
 & & \quad + \sum_{p = n}^{p = t-1} H(g,t-1,e,f,p,D) \nonumber  \\
 & & \quad + \sum_{j=1}^{|D|} H(g,t-1,e,f,d_j+ n-1,D - \{d_j\}).  \nonumber 
\end{eqnarray}
\end{corollary}

\section{Multirooted hypermaps}
\label{multi:hypermap:sec}

For $\rho \geq 1$ a $\rho$-rooted hypermap is a hypermap in which a sequence of
$\rho$ darts with pairwise distinct initial vertices is distinguished. A
multirooted hypermap is a $\rho$-rooted hypermap for some $\rho \geq 1$. This section addresses
the enumeration of multirooted hypermaps.

\begin{thm}[Recurrence between numbers of multirooted hypermaps]
\label{hyp:multi:rec:number:cor}
Let $H_m(g,t,f,e,D)$ be the number of multirooted hypermaps of genus $g$ with
$t$ darts, $f$ faces and $e$ hyperedges such that $D$ is the list of degrees
of the distinguished vertices. Then $H_m(0,0,1,0,[\,]) = 1$ and if $t \geq
1$, then

\begin{eqnarray}
& & H_m(g,t,f,e,n.D) = \nonumber \\
 & & \quad \sum_{
    \begin{scriptsize}
    \begin{array}{c}
     g_1 + g_2 = g \\
     t_1 + t_2 = t - 1 \\
     f_1 + e_2 = e \\
     f_2 + e_1 = f \\
     n_1+n_2 = n - 1 \\
     D_1 \subseteq D 
    \end{array}
    \end{scriptsize}
   } 
   H_m(g_1,t_1,f_1,e_1,n_1.D_1) \, H_m(g_2,t_2,f_2,e_2,n_2.(D-D_1)) \nonumber 
   \\
 & & \quad + \ \delta_{n \geq 3} \delta_{g \geq 1} 
         \sum_{p=1}^{n-2} H_m(g-1,t-1,e,f,(n-1-p).p.D)
         \label{rec:multi:hypermap}
         \\
 & & \quad + \sum_{p = n}^{p = t-1} H_m(g,t-1,e,f,p.D) \nonumber  \\
 & & \quad + \sum_{j=1}^{|D|} d_j \, H_m(g,t-1,e,f,(d_j+ n-1).(D - \{d_j\})). 
 \nonumber
\end{eqnarray}
\end{thm}

\begin{proof}
A multirooted hypermap is similar to a sequenced rooted hypermap except that for
each distinguished non-root vertex a dart starting from it is distinguished. If
the degree of the $j$th distinguished vertex is $d_j$, then there are $d_j$ ways of
distinguishing a dart of this vertex. It follows that for each
sequenced rooted hypermap, there are $\Pi_{j = 1}^{|D|} d_j$ multirooted
hypermaps. Let $H_m(g,t,f,e,D)$ be the number of multirooted hypermaps of
genus $g$ with $t$ darts, $f$ faces and $e$ hyperedges such that such that
$D$ is the list of degrees of the initial vertex of the distinguished
darts. Then

\begin{equation}
H_m(g,t,f,e,n.D) = H(g,t,f,e,n,D) \, \Pi_{j = 1}^{|D|} d_j.
\label{seq:multi:hyp:rel}
\end{equation} 

Solving~(\ref{seq:multi:hyp:rel}) for $H(g,t,f,e,n,D)$ and substituting
into~(\ref{rec:seq:hypermap}) proves the theorem.
\end{proof}

For $\rho \geq 1$ let
\begin{equation}
\textit{H}_{{g}}(v_{1},\ldots,v_{\rho},x,y,u,z) = \sum_{
 \begin{tiny}
  	\begin{array}{c}
   t \geq 0, f \geq 1, e \geq 0 \\ 
   d_1 \geq 1, \ldots, d_{\rho} \geq 1\\
   v = t + 2(1 – g) - e - f
  \end{array}
 \end{tiny}
 } \
 H_{m}(g,t,f,e,D_{1..\rho}) v_{1..\rho}^{D_{1..\rho}} 
 x^{f} y^{e} u^{v} z^{t}
\label{hyp:multi:series:def}
\end{equation}
be the generating function that counts multirooted hypermaps of genus $g$ with
$\rho$ distinguished darts if $g \geq 0$, and 0 otherwise. For $1 \leq i \leq
\rho$, the exponent $d_i$ of the variable $v_i$ in this series is the degree of
the initial vertex of the $i$-th  distinguished dart. The exponent $f$ of the
variable $x$ is the number of faces, the exponent $e$ of the variable
$y$ is the number of hyperedges, the exponent $t$ of the variable $z$ is the
number of darts and the exponent $v$ of the variable $u$ is the number of
vertices ($v$ is computable from the other
parameters by Formula~(\ref{genus:formula:hyper})).

\begin{corollary}[Functional equations for multirooted hypermaps]
\label{hyp:multi:eq:cor}
For $g \geq 0$ and $\rho \geq 1$ the generating functions $H_{{g}}$ of
multirooted hypermaps of genus $g$ are defined by the following functional
equations:
\begin{eqnarray}
 & & H_{{g}}(v_{1},W,x,y,u,z) = \nonumber \\
& & \quad 
 \frac{y v_{1} z}{x u}
   \sum _{j=0}^{g}
     \sum _{X \subseteq W}
       H_{{j}}(v_{1},X,y,x,u,z) H_{{g-j}}(v_{1},W-X,x,y,u,z)  \nonumber 
\\
& & \quad + \ {\frac
      {v_{1} z}
           {u}} 
    H_{{g-1}}(v_{1},v_{1},W,y,x,u,z)        \label{H:functional:eq} \\
& & \quad + \ {\frac 
       {v_{1} u z}
      {v_{1}-1}}
    \left(H_{{g}}(v_{1},W,y,x,u,z)- H_{{g}}(1,W,y,x,u,z) \right)
    \nonumber
\\
& & \quad +\  v_{1} u z \nonumber \\
& & \quad\phantom{+} \ \  
     \sum _{j=2}^{j={\rho}}v_{j}
        {\frac {\partial }{\partial v_{j}}}
        \left( v_{j}
          {\frac 
  {H_{{g}}(v_{j},W-\{v_{j}\},y,x,u,z)- H_{{g}}(v_{1},W-\{v_{j}\},y,x,u,z)}
                       {v_{j}-v_{1}}}
    \right) \nonumber
\\
&  & \quad + \ x u \delta_{g=0}\delta_{\rho=1}, \nonumber
\end{eqnarray}
where $W=v_2,\ldots,v_{\rho}$.
\end{corollary}

\begin{proof}
By summation according to~(\ref{hyp:multi:series:def}) of the recurrence between
numbers of multirooted hypermaps from Theorem~\ref{hyp:multi:rec:number:cor}.
\end{proof}

By vertex-hyperedge duality, we have
\begin{equation}
H_{g}(v_{1},W,y,x,u,z) = H_{g}(v_{1},W,x,y,u,z)
+\delta_{g=0}\delta_{\rho=1}(yu - xu)
\label{vertex:edge:dual:multi:eq}
\end{equation}
 and thus another functional equation without $x,y$ swaps is:

\begin{eqnarray}
& & H_{{g}}(v_{1},W,x,y,u,z) = \nonumber
\\
& & \quad 
 \frac{y v_{1} z}{x u}
   \sum _{j=0}^{g}
     \sum _{X \subseteq W}
       \left(
       \begin{array}{l}
       \left(H_{j}(v_{1},X,x,y,u,z)+\delta_{j=0}\delta_{|X|=0}(yu- xu)
       \right)\\
 H_{{g-j}}(v_{1},W-X,x,y,u,z)
       \end{array}
       \right)  \nonumber 
\\
& & \quad + \ {\frac
      {v_{1} z}
           {u}} 
    H_{{g-1}}(v_{1},v_{1},W,x,y,u,z)        \label{H:functional:dual:z:eq} \\
& & \quad +\ {\frac 
       {v_{1} u z}
      {v_{1}-1}}
    \left( H_{{g}}(v_{1},W,x,y,u,z)-H_{{g}}(1,W,x,y,u,z) \right) \nonumber
\\
& & \quad + \ v_{1} u z \nonumber \\
& & \quad\phantom{+} \ \sum _{j=2}^{j={\rho}}v_{j}
        {\frac {\partial }{\partial v_{j}}}
        \left( v_{j}
          {\frac 
  {H_{{g}}(v_{j},W-\{v_{j}\},x,y,u,z)-H_{{g}}(v_{1},W-\{v_{j}\},x,y,u,z)}
                       {v_{j}-v_{1}}}
    \right) \nonumber
\\
&  & \quad + x u \delta_{g=0}\delta_{\rho=1}. \nonumber
\end{eqnarray}

The former equation is given here for maximal generality. However, a consequence
of the genus formula (\ref{genus:formula:hyper}) is that three variables among
the four variables $x$, $y$, $u$ and $z$ are sufficient. In the remainder of the
paper we consider the generating functions 
$$H_{{g}}(v_{1},W,x,y,u) =
H_{{g}}(v_{1},W,x,y,u,1)$$ with one fewer variable. They are defined by the
following functional equations:

\begin{eqnarray}
& & H_{{g}}(v_{1},W,x,y,u) = \nonumber
\\
& & \quad 
 \frac{y v_{1}}{x u}
   \sum _{j=0}^{g}
     \sum _{X \subseteq W}
       \left(H_{j}(v_{1},X,x,y,u)+\delta_{j,0}\delta_{|X|,0}(yu- xu)
       \right)
 H_{{g-j}}(v_{1},W-X,x,y,u)  \nonumber 
\\
& & \quad + \ {\frac
      {v_{1}}
           {u}} 
    H_{{g-1}}(v_{1},v_{1},W,x,y,u)        \label{H:functional:dual:eq} \\
& & \quad + \ {\frac 
       {v_{1} u}
      {v_{1}-1}}
    \left( H_{{g}}(v_{1},W,x,y,u)-H_{{g}}(1,W,x,y,u) \right) \nonumber
\\
& & \quad + \  v_{1} u 
     \sum _{j=2}^{j={\rho}}v_{j}
        {\frac {\partial }{\partial v_{j}}}
        \left( v_{j}
          {\frac 
  {H_{{g}}(v_{j},W-\{v_{j}\},x,y,u)-H_{{g}}(v_{1},W-\{v_{j}\},x,y,u)}
                       {v_{j}-v_{1}}}
    \right) \nonumber
\\
&  & \quad +\ x u \delta_{g=0}\delta_{\rho=1}. \nonumber
\end{eqnarray}

For $g,{\rho} \neq 0,1$, after grouping in the left-hand side the terms
containing $H_{g}(v_{1},W,x,y,u)$ in (\ref{H:functional:dual:eq}), one gets
\begin{eqnarray}
& & \frac{A(v_{1},x,y,u)}{v_1} H_{g}(v_{1},W,x,y,u) 
  = \nonumber \\
& & \quad 
 x (1-v_{1}) 
   \sum _{j=0}^{g}
     \sum _{
      \begin{tiny}
      \begin{array}{c}
       X \subseteq W\\
       (j,X) \neq (0,[\,])\\
       (j,X) \neq (g,W)
      \end{array}
      \end{tiny}
     }
        H_{j}(v_{1},X,x,y,u)
 H_{{g-j}}(v_{1},W-X,x,y,u)  
\nonumber \\
& & \quad + \ {\frac
      {1-v_{1}}
           {u}} 
    H_{{g-1}}(v_{1},v_{1},W,x,y,u)       
+  u H_{{g}}(1,W,x,y,u) 
\nonumber \\
& & \quad
 + \ u T_{g}(v_{1},W,x,y,u)  \label{H:recurrence:eq}
\end{eqnarray}
with
\begin{equation}
A(v,x,y,u) = 
v u+(1-v)(1-y v+x v-2v H_0(v,x,y,u)/u)
\label{A:def:eq}
\end{equation}
and
\begin{eqnarray}
& & T_{g}(v_{1},W,x,y,u) =  \nonumber \\
& & \quad   
 (1-v_1) \sum_{j=2}^{j={\rho}} v_{j}
  \frac {\partial }{\partial v_{j}}
   \left( \frac{v_{j}}{v_{j}-v_{1}}
    \left(
     \begin{array}{l}
     H_{{g}}(v_{j},W-\{v_{j}\},x,y,u)\\
     -H_{{g}}(v_{1},W-\{v_{j}\},x,y,u)
     \end{array}
\right)
    \right).
    \label{T:eq}
\end{eqnarray}

\section{Rooted hypermap generating functions}
\label{rooted:hyper:series:sec}

Let $h_g(v,e,f)$ be the number of rooted genus-$g$ hypermaps with 
$v$ vertices, $e$ hyperedges and $f$ faces. Let 
\begin{equation}
H_g(x,y,u) = \sum_{v,e,f \geq 1} h_g(v,e,f) x^v y^e u^f
\label{xuHg:eq}
\end{equation}
be the ordinary generating function for counting rooted hypermaps on the
orientable surface of genus $g \geq 0$, where the exponent of variable $x$ is the
number of vertices, the exponent of variable $y$ is the number of hyperedges, and
the exponent of variable $u$ is the number of faces. 

Rooted hypermaps being $1$-rooted hypermaps,
\begin{equation}
H_g(x,y,u) = H_g(1,x,y,u),
\end{equation}
where $H_g(v_{1},\ldots,v_{\rho},x,y,u)$ is the generating function counting
$\rho$-rooted genus-$g$ hypermaps defined in
Section~\ref{multi:hypermap:sec} for $\rho \geq 1$. 

We first recall in Section~\ref{basic:case:proof:sec} a known parametric
expression of the generating function 
that counts rooted planar hypermaps. Then we explain in
Section~\ref{pattern:positive:genus:sec} how to solve the functional equation of
the generating functions $H_g(x,y,u)$ that count rooted hypermaps with a given
positive genus $g$.

\subsection{Rooted planar hypermaps}
\label{basic:case:proof:sec}

The following proposition is a reformulation of \cite[Theorem~3]{Arq86a}, with
the correspondence $s = x$, $f = u$ and $a = y$ for variables, $\lambda =
p$, $\mu = q$ and $\nu = r$ for parameters, and $H_0 = sf(1+J)$ for generating
functions.

\begin{proposition}{\cite{Arq86a}}
\label{rooted:hyp:param:planar:prop}
The ordinary generating function $H_{0}(x,y,u)$ that counts rooted planar
hypermaps by number of vertices (exponent of $x$), hyperedges (exponent of $y$)
and faces (exponent of $u$) is the unique solution of the following parametric
system:
\begin{eqnarray}
H_0(x,y,u) = 1+p q r (1-p-q-r) \label{H0:param:eq}
\end{eqnarray}
with 
\begin{equation}
\left\{
\begin{array}{l}
x  =  p (1-q-r) \\
u  =  q (1-p-r) \\
y  =  r (1-p-q).
\end{array}
\right.   \label{xuy:param:eq}
\end{equation}
\end{proposition}

\begin{proof}
The generating function $H_{0}(v,x,y,u)$ that counts rooted planar hypermaps
(genus $0$) by number of vertices (exponent of $x$), hyperedges (exponent of
$y$), faces (exponent of $u$) and degree of the root vertex (exponent of $v$)
satisfies the functional equation
\begin{eqnarray}
H_{0}(v,x,y,u) & = & \frac{y v}{xu} \left(H_{0}(v,x,y,u) + yu - xu\right)
H_{0}(v,x,y,u) \nonumber \\ & &  
+\frac{v u}{v-1}
    \left( H_{0}(v,x,y,u)-H_{0}(1,x,y,u) \right) + xu \label{H:0:1:eq}
\end{eqnarray}
obtained by instantiation of (\ref{H:functional:dual:eq}) with $g=0$,
${\rho}=1$ and $v_1 = v$.

This equation can be solved by the \emph{quadratic method}~\cite[page
515]{FS09}. The idea is to define auxiliary functions $A(v,x,y,u)$ and
$B(v,x,y,u)$ by (\ref{A:def:eq})
and
\begin{eqnarray}
B(v,x,y,u) & = & A(v,x,y,u)^2  \label{AsqrB:eq}
\end{eqnarray}
and look for a function $V(x,y,u)$ such that 
\begin{equation}
A(V(x,y,u),x,y,u) = 0, \label{AV0:eq}
\end{equation}
 implying that $B(V(x,y,u),x,y,u) = 0$ and
$\partial_v B(v,x,y,u)_{|v = V(x,y,u)} = 0$.

We get from (\ref{H:0:1:eq}), (\ref{A:def:eq}) and (\ref{AsqrB:eq}) that
\begin{eqnarray}
 & & B(v,x,y,u) = \nonumber \\
 & & \quad 1-2 y v-2 x v-2 v^3 y-2 v^3 x-2 v^2 u+v^4 y^2-2
 v^3 y^2 +y^2 v^2+v^4 x^2
\nonumber \\
& & \quad -2 v^3 x^2+x^2 v^2+v^2 u^2+4 v^3 y x -2 y v^2 x-2 y v^2 u+2 v^3 y u-2
v^4 y x 
\nonumber \\
& & \quad -2 v^3 x u+2 x v^2 u +4 v^2 x+4 v^2 y +2 v u+4 v^3 H_0(1,x,y,u)
\nonumber \\
& & \quad -4 v^2 H_0(1,x,y,u)-2 v+v^2. \label{B:def:eq}
\end{eqnarray}
The constraints  $B(V(x,y,u),x,y,u) = 0$ and
$\partial_v B(v,x,y,u)_{|v = V(x,y,u)} = 0$ respectively are
\begin{eqnarray}
1-2 y V-2 x V-2 V^3 y-2 V^3 x-2 V^2 u+V^4 y^2-2 V^3 y^2 +y^2 V^2
\nonumber \\
 +\ V^4 x^2 -2 V^3 x^2+x^2 V^2+V^2 u^2+4 V^3 y x -2 y V^2 x-2 y V^2 u
\nonumber \\
+\ 2 V^3 y u x -2 V^4 y -2 V^3 x u+2 x V^2 u +4 V^2 x+4 V^2 y +2 V u \nonumber \\
+\ 4 V^3 
H_0(1,x,y,u) -4 V^2  H_0(1,x,y,u)-2 V+V^2 & = & 0  \label{V:def:eq:1}
\end{eqnarray}
and
\begin{eqnarray}
-2+8 y V+8 x V+4 V^3 y^2-6 y^2 V^2+4 V^3 x^2-6 x^2 V^2-6 V^2 x
\nonumber \\
-6 V^2 y-4 V u + 2 y^2 V+2 x^2 V+2 V u^2-4 y V u+4 x V u-4 y V x
 \nonumber \\
+\ 12 y V^2 x+6 y V^2 u -8 V^3 y x -6 x V^2 u +2 V-2 x-2 y+2 u & = & 0.
\label{V:def:eq:2}
\end{eqnarray}

It can be checked that both equations are satisfied by 
\begin{equation}
V = 1/(1-q) \label{V:def:eq}
\end{equation} with $x$, $u$, $y$ and $H_0(1,x,y,u)$ related to $p$, $q$ and $r$
by (\ref{xuy:param:eq}) and
(\ref{H0:param:eq}).
\end{proof}

\subsection{Rooted hypermaps with positive genus}
\label{pattern:positive:genus:sec}

The following additional notations are used in this section. Let $\rho$ be a
positive integer. Let $H_j[n_1,\ldots,n_{\rho}]$ denote the partial derivative
of the function $H_j(v_1,\ldots,v_{\rho},x,y,u)$ with respect to the variables
$v_1$, \ldots, $v_{\rho}$ to the respective orders $n_1$, \ldots, $n_{\rho}$,
computed at $v_1 = \ldots = v_{\rho} = V$.
The abbreviation $[\rho]$ denotes the list $[2, \ldots, \rho]$ if $\rho \geq 2$
and the empty list $[\,]$ if $\rho=1$. The abbreviation $N_{[\rho]}$ denotes the
list $[n_2, \ldots, n_{\rho}]$. For any sublist $X \subseteq [\rho]$ of
$[\rho]$, $[\rho]-X$ denotes the sublist of the elements of $[\rho]$ that are
not in $X$, $N_X$ denotes the list of those $n_i$ in $N_{[\rho]}$ such that $i$
is in $X$ and $N_j$ denotes the list $[n_2, \ldots, n_{j-1}, n_{j+1},
\ldots,n_{\rho}]$.

\subsubsection{Equation for rooted hypermaps and recurrence relations}
\label{proof:recur:rel:sec}

The special case of Formula (\ref{H:recurrence:eq}) for $g \geq 1$, $\rho = 1$
 and $v_1 = V$ is the following formula:

\begin{eqnarray*}
 & & u H_g(1,x,y,u) = \\
 & & \quad 
 (V-1) \left(x \sum_{j = 1}^{g-1} H_j(V,x,y,u) H_{g-j}(V,x,y,u)
 +H_{g-1}(V,V,x,y,u)/u\right)
\label{Hg:V:eq}
\end{eqnarray*}

i.e.

\begin{eqnarray}
u H_g(1,x,y,u) & = & (V-1) \left(x \sum_{j = 1}^{g-1}
H_j[0] H_{g-j}[0] +H_{g-1}[0,0]/u\right).
\label{Hg:V:0:eq}
\end{eqnarray}

In order to derive from (\ref{Hg:V:0:eq}) a value for $H_g(1,x,y,u)$, we are
looking for a value for 
$H_j[0]$, $H_{g-j}[0]$ and $H_{g-1}[0,0]$. More generally, we will
derive from
the following proposition a closed form for the expressions
$H_g[n_1,\ldots,n_{\rho}]$.

\begin{proposition}
\label{Hg:function:prop}
For $g \geq 0$, $\rho \geq 1$ and $n_1,\ldots,n_{\rho} \geq 0$ the function
$H_g[n_1,\ldots,n_{\rho}]$ is defined by
\begin{eqnarray}
& & 
\frac{(n_1+1) A[1]}{V} H_g[n_1,N_{[\rho]}]  = \nonumber \\ 
& & \quad
 \sum_{
  \begin{subarray}{c}
   i+j+k=n_1+1 \\
   i > 0,\, k < n_1 
  \end{subarray}
 }
 {n_1+1 \choose i,j} 
 \frac{\left (-1\right )^{j+1} j!}{V^{j+1}}
 A[i]
 H_{g}[k,N_{[\rho]}] \nonumber
\\ & &
\quad + \ x 
 \sum_{
  \begin{subarray}{c} 
   k+l+m=n_1+1 \\
   0 \leq j \leq g \\
   X \subseteq [{\rho}] \\
   (j,X) \neq (0,[\,]) \\
   (j,X) \neq (g,[{\rho}]) 
  \end{subarray}
 }
 {n_1+1 \choose k,l} M[m] H_{j}[k,N_{X}] H_{g-j}[l,N_{[\rho]-X}]
\nonumber
\\ 
& &
\quad
 + \frac{1}{u}
 \sum_{i+j+k=n_1+1}
  {n_1+1 \choose i,j} M[k] H_{g-1}[i,j,N_{[\rho]}] \label{HgN:eq}
\\ 
& &
\quad 
+\ u \sum_{j=2}^{{\rho}}
  \frac{(n_1+1)! n_j!}{(n_1+n_j+2)!}
  \left(
   \begin{array}{l}
       n_j F_g[n_1+n_j+2,N_j] \\
       + \frac{V(n_j+1)}{n_1+n_j+3}F_g[n_1+n_j+3,N_j]
    \end{array}
  \right), 
  \nonumber
\end{eqnarray}
where 
\begin{equation}
F_g(v_1,\ldots,v_h,x,y,u) = L(v_1) H_g(v_1,\ldots,v_h,x,y,u)
\label{Fg:def:eq}
\end{equation}
for $h \geq 1$, $M(v) = 1-v$ and $L(v)=v(1-v)$.
\end{proposition}

\begin{proof}
Equation~(\ref{HgN:eq}) is 
obtained from Eq.~(\ref{H:recurrence:eq}) as follows:
\begin{enumerate}
\item Partial derivation of~(\ref{H:recurrence:eq}) with respect to the
variables $v_1$, $v_2$, \ldots, $v_{\rho}$ to the respective orders $n_1+ 1$,
$n_2$, \ldots, $n_{\rho}$.
\item \label{step2} Evaluation of this differential equation at $v_1 = \cdots =
v_{\rho} = V$. The function $H_g[n_1+1,\ldots,n_{\rho}]$ is multiplied by $A[0]$ in the
resulting equation, and $A[0]$ is known to be zero (\ref{AV0:eq}).
The functions $T_g[\ldots]$ are replaced by expressions with the functions $F_g[\ldots]$ thanks to
Lemma~\ref{T:F:lemma} below.
\item In the left-hand side of the resulting equation, isolation of the single
term involving the function $H_g[n_1,\ldots,n_{\rho}]$.
\end{enumerate}
By inspection one can check that the right-hand side of (\ref{HgN:eq}) depends
only on some functions $H_g[k,n_2,\ldots,n_{\rho}]$ with $k < n_1$, some
functions $H_g[n'_1,\ldots,n'_{\rho'}]$ with $\rho' < \rho$ and some functions
$H_j[\ldots]$ for $j < g$. Thus, (\ref{HgN:eq}) in a recursive definition of
the family of functions $H_g[n_1,\ldots,n_{\rho}]$ for $g \geq 0$, $\rho \geq 1$
and $n_1,\ldots,n_{\rho} \geq 0$.
\end{proof}

The following lemma relates the partial derivatives of $T_g$ at $v=V$ with the
ones of $F_g$.

\begin{lemma}
\label{T:F:lemma}
For $\rho \geq 2$ and $g, n_1, \ldots, n_{\rho} \geq 0$,
\begin{eqnarray}
& & T_g[n_1+1,N_{[\rho]}]   = \nonumber \\
 & & \quad
\sum _{j=2}^{j={\rho}} 
   \frac{(n_1+1) ! n_j !}{(n_1+n_j+2)!} 
   \left(
   \begin{array}{l}
   n_j F_g[n_1+n_j+2,N_j]  \\
     + \frac{V (n_j+1)}{n_1+n_j+3} F_{g}[n_1+n_j+3,N_j]
   \end{array}
    \right) \label{T:Hg:rel:eq}.
\end{eqnarray}
\end{lemma}

\begin{proof}
We can easily prove that
\begin{eqnarray}
\frac{\partial}{\partial v_{j}}
        \left[
          {\frac 
  {(v_{j}-v_1) H_{{g}}(v_{1},[\rho]-\{v_{j}\},x,y,u)}
                       {v_{j}-v_{1}}}
    \right] = 0.
\end{eqnarray}
Then, $T_g(v_1,\ldots,v_{\rho},x,y,u)$ equals
\begin{eqnarray}
 \sum _{j=2}^{j={\rho}} v_{j}
  {\frac {\partial }{\partial v_{j}}}
   \left( \left(v_{j}-v_{1}\right)^{-1}
    \left(
     \begin{array}{l}
  v_{j} (1-v_1) H_{{g}}(v_{j},[\rho]-\{v_{j}\},x,y,u)\\
  - \ v_{1} (1-v_1) H_{{g}}(v_{1},[\rho]-\{v_{j}\},x,y,u)
     \end{array}
     \right)
    \right).
\end{eqnarray}
It also holds that
\begin{eqnarray}
\frac{\partial^{n_1+1}}{\partial v_{1}^{n_1+1}}
        \left[
          {\frac 
  {v_j (v_{j}-v_1) H_{{g}}(v_{j},[\rho]-\{v_{j}\},x,y,u)}
                       {v_{j}-v_{1}}}
    \right] & = & 0,
\end{eqnarray}
so that $\frac{\partial^{n_1+1}}{\partial v_{1}^{n_1+1}}
T_g(v_1,\ldots,v_{\rho},x,y,u)$ equals
\begin{eqnarray}
\sum _{j=2}^{j=r} v_{j}
 \frac{\partial^{n_1+2}}{\partial v_{1}^{n_1+1}\partial v_{j}}
 \left(
  \left(v_{j}-v_{1}\right)^{-1}
  \left(
   \begin{array}{l}
   v_{j} (1-v_j) H_{{g}}(v_{j},[\rho]-\{v_{j}\},x,y,u) \\
   - \ v_{1} (1-v_1) H_{{g}}(v_{1},[\rho]-\{v_{j}\},x,y,u)
   \end{array}
  \right)
 \right)
\end{eqnarray}
i.e.
\begin{eqnarray}
\sum _{j=2}^{j={\rho}} v_{j}
        {\frac {\partial^{n_1+2}}{\partial v_{1}^{n_1+1}\partial v_{j}}}
        \left(
          {\frac 
  {F_{g}(v_{j},[\rho]-\{v_{j}\},x,y,u)-
  F_{g}(v_{1},[\rho]-\{v_{j}\},x,y,u)}
                       {v_{j}-v_{1}}}
    \right).
\end{eqnarray}
Formula (\ref{T:Hg:rel:eq}) is a consequence of
\begin{equation}
\frac{\partial^{n_1+n_2}}{\partial x_1^{n_1}\partial x_{2}^{n_2}}
        \left(
          \frac{\psi(x_1)-\psi(x_2)}
                       {x_{1}-x_{2}}
    \right)_{x_1=x_2=a} = \frac{n_1 ! n_2 !}{(n_1+n_2+1)!}
    \psi^{(n_1+n_2+1)}(a).
\end{equation}
\end{proof}
\noindent The formula
\begin{eqnarray}
F_g[n,N] = \sum_{k+l=n} {n \choose k} L[k] H_g[l,N] \label{F:Hg:rel:eq}
\end{eqnarray} 
is an easy consequence of (\ref{Fg:def:eq}). Thus the right-hand side of
(\ref{HgN:eq}) only depends on some functions $H_g[k,\ldots,n_{\rho}]$ with $k <
n_1$, some functions $H_g[n'_1,\ldots,n'_{\rho'}]$ with $\rho' < \rho$, some
functions $H_j[\ldots]$ for $j < g$ and some functions $A[i]$. A relation
between $A[i]$ and some functions $H_0[j]$ is established in
Section~\ref{Ak:sec}.

\subsubsection{Case $g=0$ and ${\rho}=1$}
\label{Ak:sec}

The function $A[i]$ can be related to some functions $H_0[j]$  as follows:
With $M(v) = 1-v$ and $L(v)=v(1-v)$, Eq.~(\ref{A:def:eq}) is
\begin{equation}
A(v,x,y,u) = v u+M(v)+L(v)(-y+x -2x  H_0(v,x,y,u)). \label{A:H0:v:eq}
\end{equation}
Its instantiation at $v=V$ gives 
\begin{equation}
H_0[0] = \frac{1-q}{1-q-r}.
\end{equation}
For $k \geq 1$, the $k$-th partial derivative of (\ref{A:H0:v:eq}) in $v$ is
\begin{eqnarray}
\frac{\partial^k}{\partial v^k} A(v,x,y,u) & = &
\frac{\partial^k}{\partial v^k}(v u)+\frac{\partial^k}{\partial
v^k} M(v) \nonumber \\
& & +\frac{\partial^k}{\partial
v^k}\left[L(v)(-y+x -2x  H_0(v,x,y,u))\right]
\end{eqnarray}
and its instantiation in $v=V$ is
\begin{eqnarray}
 A[k] & = & \frac{\partial^k}{\partial v^k}(v u)_{|v=V} + M[k] \nonumber \\ 
& & +
\sum_{i+j=k} {k \choose i}
L[i]\left(\frac{\partial^j}{\partial v^j}(-y+x-2x H_0(v,x,y,u))_{|v=V}\right).
\label{A:H0:deriv:eq}
\end{eqnarray}

\noindent Solving (\ref{A:H0:deriv:eq}) for $k = 1$ gives
\begin{eqnarray}
H_0[1] & = & \frac{(1-q)^2 (A[1]+1-p-q-r)}{2 p q (1-q-r)}.
\end{eqnarray}

\noindent For $k \geq 2$, one gets
\begin{eqnarray*}
A[k] & = & 
-2x \sum_{i+j=k} {k \choose i} L[i] H_0[j]
\end{eqnarray*}
since $M[k] = 0$, i.e.
\begin{eqnarray}
A[k] & = &
 -2x \left(L[0] H_0[k]+k L[1] H_0[k-1]+\frac{k(k-1)}{2}
 L[2] H_0[k-2]\right) \label{Ak:eq}
\end{eqnarray}
since  $L[k] = 0$ if $k \geq 3$.

\section{Explicit formulas for small genera}
\label{hyp:series:explicit:small:genus:sec}
This section proposes explicit parametric expressions for the generating
functions that count rooted hypermaps of small positive genus. In
Section~\ref{rooted:hyper:series:3:sec} we count by number of vertices,
 hyperedges and faces; the number of darts can be obtained from these parameters
by Formula~(\ref{genus:formula:hyper}). In
Section~\ref{rooted:hyper:series:dart:sec} we count by number of darts alone.

\subsection{Rooted hypermap series enumerated with three parameters}
\label{rooted:hyper:series:3:sec}
For $g = 1,\ldots,5$ we have computed an explicit expression of $H_g(x,y,u)$
parameterized by $p$, $q$ and $r$, with
$x = p (1-q-r)$, $u = q (1-p-r)$  and $y = r (1-p-q)$, by application of
formulas in Section~\ref{rooted:hyper:series:sec}.
For $g \geq 3$, the 
 expressions are too large to be included in the present text, but a Maple file
 with all the generating functions up to genus 5 is available from the first
 author on request.

A parametric expression of $H_1(x,y,u)$ is
\begin{equation}
H_1(x,y,u) = 
\frac{p \; q \; r \; (1 -p)
 \; (1 -q) \; (1- r) \;}
{\left[(1-p-q-r)^2-4pqr\right]^2}. \label{H1:eq}
\end{equation}
This expression can be derived from \cite[Theorem 3]{Arq87}, with the
correspondence $s = x$, $f = u$, and $a = y$ between variables and the
correspondence $H_1(x,y,u) = x u K_1(1,x,y,u)$ between generating functions.

A parametric expression of $H_2(x,y,u)$ is
\begin{equation}
H_2(x,y,u) = 
\frac{
 p \; q \; r \; (1 -p)
 \; (1 -q) \; (1- r) \; P_2(p,q,r)}{
  \left[(1-p-q-r)^2-4pqr\right]^7}
\end{equation}
where 
\begin{align*}
P_2(p,q,r) = &~ 76 p^6 q^2 r^2- 8 p^4 q^4 r^2 - 8 p^4 q^2 r^4 + 76 p^2 q^6 r^2 
              - 8 p^2 q^4 r^4 + 76 p^2 q^2 r^6  \\
   &+  40 p^7 q r - 76 p^6 q^2 r - 76 p^6 q r^2 - 112 p^5 q^3 r - 228 p^5 q^2
   r^2 - 112 p^5 q r^3 \\
   &+ 8 p^4 q^4 r  + 16 p^4 q^3 r^2 + 16 p^4 q^2 r^3 + 8 p^4 q r^4 - 112 p^3 q^5 r + 16 p^3 q^4 r^2 \\
   &+ 40 p^3 q^3 r^3  + 16 p^3 q^2 r^4 
   - 112 p^3 q r^5 - 76 p^2 q^6 r - 228 p^2 q^5 r^2 \\
   &+ 16 p^2 q^4 r^3 +
   16 p^2 q^3 r^4 - 228 p^2 q^2 r^5 - 76 p^2 q r^6 
    + 40 p q^7 r - 76 p q^6 r^2 \\
   &- 112 p q^5 r^3 + 8 p q^4 r^4 - 112 p q^3 r^5
    - 76 p q^2 r^6 + 40 p q r^7 + p^8
    -  20 p^7 q \\
   &- 20 p^7 r - 35 p^6 q^2 - 64 p^6 q r - 35 p^6
   r^2 + 56 p^5 q^3 + 396 p^5 q^2 r + 396 p^5 q r^2 \\
   &+ 56 p^5 r^3 + 140 p^4 q^4 + 264
   p^4 q^3 r + 393 p^4 q^2 r^2 + 264 p^4 q r^3 + 140 p^4 r^4 \\
   &+ 56 p^3 q^5 
   +  264 p^3 q^4 r - 92 p^3 q^3 r^2 - 92 p^3 q^2 r^3 + 264 p^3 q r^4 + 56 p^3
   r^5 \\
   &- 35 p^2 q^6 + 396 p^2 q^5 r 
    +  393
   p^2 q^4 r^2 - 92 p^2 q^3 r^3 + 393 p^2 q^2 r^4 + 396 p^2 q r^5 \\
   &- 35 p^2 r^6
   - 20 p q^7 - 64 p q^6 r
    +  396 p q^5 r^2 
   + 264 p q^4 r^3 + 264 p q^3 r^4\\
   &+ 396 p q^2 r^5 - 64 p q r^6 - 20 p r^7 +
   q^8 - 20 q^7 r 
   -  35 q^6 r^2 + 56 q^5 r^3\\
   &+ 140 q^4 r^4 + 56 q^3 r^5 - 35 q^2 r^6
   - 20 q r^7 + r^8 + 6 p^7 + 105 p^6 q + 105 p^6 r \\
   &+  21 p^5 q^2 - 116 p^5 q r +
   21 p^5 r^2 - 420 p^4 q^3 - 821 p^4 q^2 r - 821 p^4 q r^2 \\
   &- 420 p^4 r^3 -
   420 p^3 q^4
    -  648 p^3 q^3 r - 316 p^3 q^2 r^2 - 648 p^3 q r^3 - 420 p^3 r^4 \\
   &+ 21
   p^2 q^5 - 821 p^2 q^4 r - 316 p^2 q^3 r^2 
   -  316 p^2 q^2 r^3 - 821 p^2 q r^4 + 21 p^2 r^5 \\
   &+ 105 p q^6 - 116 p q^5 r -
   821 p q^4 r^2 - 648 p q^3 r^3 -  821
   p q^2 r^4 - 116 p q r^5 \\
   &+ 105 p r^6 + 6 q^7 + 105 q^6 r + 21 q^5 r^2 - 420
   q^4 r^3 - 420 q^3 r^4 + 21 q^2 r^5 \\
   &+  105 q r^6 + 6 r^7 - 49 p^6 - 189 p^5 q -
   189 p^5 r  + 315 p^4 q^2 + 479 p^4 q r \\
   &+ 315 p^4 r^2 + 910 p^3 q^3 
    +  1162
   p^3 q^2 r + 1162 p^3 q r^2 + 910 p^3 r^3 + 315 p^2 q^4 \\
%
   & + 1162 p^2 q^3 r +
   720 p^2 q^2 r^2 + 1162 p^2 q r^3 
   +  315 p^2 r^4 - 189 p q^5 + 479 p q^4 r \\
   &+ 1162 p
   q^3 r^2 + 1162 p q^2 r^3 + 479 p q r^4 - 189 p r^5 - 49 q^6
  - 189 q^5 r \\
  &+ 315
   q^4 r^2 + 910 q^3 r^3 + 315 q^2 r^4 - 189 q r^5 - 49 r^6 + 112 p^5 + 70 p^4 q
   \\ 
  &+ 70 p^4 r
    -  770 p^3 q^2 - 876 p^3 q r - 770 p^3 r^2 - 770 p^2 q^3 - 1380 p^2
   q^2 r \\
  &- 1380 p^2 q r^2 - 770 p^2 r^3 + 70 p q^4 
  -  876 p q^3 r - 1380 p q^2
   r^2 - 876 p q r^3 \\
  &+ 70 p r^4 + 112 q^5 + 70 q^4 r - 770 q^3 r^2 - 770 q^2
   r^3 +  70 q r^4 + 112 r^5 \\
  &- 105 p^4 + 210 p^3 q + 210 p^3 r + 735 p^2 q^2 +
   1034 p^2 q r + 735 p^2 r^2 + 210 p q^3 \\
  &+  1034 p q^2 r + 1034 p q r^2 + 210 p r^3 -
   105 q^4 + 210 q^3 r + 735 q^2 r^2 + 210 q r^3 \\
  &- 105 r^4 + 14 p^3
    -  315 p^2 q
   - 315 p^2 r - 315 p q^2 - 672 p q r - 315 p r^2 + 14 q^3 \\
  & - 315 q^2 r - 315
   q r^2 + 14 r^3  
  +  49 p^2 + 175 p q + 175 p r + 49 q^2 + 175 q r \\
  & + 49 r^2 - 36 p
   - 36 q - 36 r + 8.
\end{align*}

\noindent Remark: For $g = 0$, the formula
\begin{eqnarray}
H_0(x,y,u) = p q r (1-p-q-r) \label{H0:param:eq:2}
\end{eqnarray}
can be derived from \cite{Arq86a}, with the correspondence $s = x$, $f = u$, and
$a = y$ between variables and the correspondence $H_0(x,y,u) = x
u K_0(1,x,y,u)$ between generating functions.

\subsection{Rooted hypermap series enumerated by number of darts} 
\label{rooted:hyper:series:dart:sec}
Let $H_{g}(z)$ be the ordinary generating function of rooted hypermaps
on the orientable surface of genus $g \geq 0$, where the exponent of variable $z$
is the number $d$ of darts.

\subsubsection{Generating functions}
For $g$ from $0$ to $6$, a parametric expression of $H_g(z)$, where $z =
\tau(1-2\tau)$ and $\tau = 0$ when $z = 0$, is
\begin{eqnarray}
H_0(z) & = & \frac{\tau^3 \; (1-3 \; \tau)}{z^2},
\label{H0mu:eq}
\\
H_1(z)  & = &  \frac{ \tau^3}{(1- \tau) \; (1-4 \;
\tau)^2}, \label{H1mu:eq}
\\
H_2(z)  & = &  \frac{ 4 \; z^2 \; \tau^3
\; (51 \; \tau^4-77 \; \tau^3+48 \; \tau^2-15
\; \tau+2) }{(1- \tau)^5  \;
(1-4 \; \tau)^7}, \label{H2mu:eq}
\\
H_3(z)  & = &  \frac{4 \; z^4 \; \tau^3
\; P_3(z) }{(1-
\tau)^9 \; (1-4 \;  \tau)^{12}  },   \label{H3mu:eq}
\\
H_4(z)  & = &  \frac{
4 \; z^6 \; \tau^3 \; P_4(z) }{
(1- \tau)^{13} \; (1-4 \;  \tau)^{17}}, 
\label{H4mu:eq}
\\
H_5(z)  & = &  \frac{
4 \; z^{8} \; \tau^3 \; P_5(z)
}{ (1- \tau)^{17}
\; (1-4 \;  \tau)^{22} } \label{H5mu:eq}
\\
H_6(z)  & = &  \frac{
4 \; z^{10} \; \tau^{3} P_6(z)
}{ (1-\tau)^{21} \; (1 - 4 \; \tau)^{27}}
 \label{H6mu:eq}
\end{eqnarray}

\noindent with
\begin{align*}
P_3(z) = &~ 28496 \; \tau^9-36888 \;
\tau^8-13164 \; \tau^7+61676 \; \tau^6 -61872
\; \tau^5+35172 \; \tau^4 \\
&-13168 \;
\tau^3+3360 \; \tau^2-552 \; \tau+45,\\
P_4(z) = & ~ 32375616
\; \tau^{14}+15509760 \; \tau^{13}-243313744
\; \tau^{12}+442844592 \; \tau^{11} \\
&-389268768
\; \tau^{10}+170357328 \; \tau^9+1281984 \;
\tau^8-53553072 \; \tau^7 \\
&+39814032 \; \tau^6 -17597520
\; \tau^5+5541192 \; \tau^4-1320920 \;
\tau^3+239697 \; \tau^2 \\
&-30456 \; \tau+2016, \\
P_5(z) = & ~ 61742404608 \; \tau^{19}+239043447552
\; \tau^{18}-1163002515456 \; \tau^{17} \\
&+1403096348736
\; \tau^{16}+338393916800 \; \tau^{15}-2962590413376
\; \tau^{14} \\
&+4243997599488 \; \tau^{13}-3552865706240
\; \tau^{12}+2000782619136 \; \tau^{11} \\
&-761565230016
\; \tau^{10}+165542511744 \; \tau^9+7568059872
\; \tau^8 \\
&-23295865824 \; \tau^7+11016156244
\; \tau^6-3336459144 \; \tau^5+761835465 \;
\tau^4 \\
& -141393220 \; \tau^3+21738240 \; \tau^2-2490480
\; \tau+151200 
\end{align*}
and
\begin{align*}
P_6(z) = &~ 178054771302400 \;
\tau^{24}+1584534210564096 \; \tau^{23}-4933663711730688 \; \tau^{22} \\
&-2073822560019456 \; \tau^{21}+28025505345377280 \;
\tau^{20}\\
&-55010184951564288 \; \tau^{19} +54283457920223232 \;
\tau^{18}\\
&-22997164994372352 \; \tau^{17}-13439214645718272 \; \tau^{16} \\
&+31734000656779264 \; \tau^{15}-29719458122609664 \;
\tau^{14}\\
&+18704646148809216 \; \tau^{13} -8736443315384448 \;
\tau^{12}\\
&+3098312828500416 \; \tau^{11}-813298324826016 \; \tau^{10}
+138473163256176 \; \tau^{9}\\
&-4043551301232 \; \tau^{8}-6580517850696 \;
\tau^{7}+2630924485729 \; \tau^{6}\\
& -626336383104 \; \tau^{5}+112079088144 \;
\tau^{4}-17314508592 \; \tau^{3}+2485496880 \; \tau^{2}\\
& -284717376 \;
\tau+17107200.
\end{align*}
 
We have also computed the generating functions for $7 \leq g \leq 11$. Their 
 expressions are too large to be included in the present text, but a Maple file
  is available from the first author on request.

A. Mednykh and R. Nedela used our formulas (\ref{H0mu:eq}) to (\ref{H3mu:eq}) to
find explicit formulas for the number of rooted hypermaps for genus $g = 0, 1,
2$ and $3$~\cite{Mednykh2016}.

\subsection{Other parameterization}
\label{zograf:sec}
In a private communication to the second author, P. Zograf suggests the  
parameterization
\begin{equation}
z = \frac{t}{(1+2 t)^2}. \label{taut:eq}
\end{equation}
After adding the condition that $t = 0$ when $z = 0$, it corresponds to
\begin{equation}
t = \frac{
1-4 z-\sqrt{1-8 z}
}{8z}.
\end{equation}

These two parameterizations are equivalent. 
The one can be transformed into the other by means of the following substitutions:
\begin{equation}
\tau = \frac{t}{1+2 t}
\end{equation}
and
\begin{equation}
t = \frac{\tau}{1-2 \tau}.
\end{equation}

By means of these substitutions, the following parametric expressions in $t$
can be obtained from the parametric expressions (\ref{H0mu:eq})-(\ref{H6mu:eq})
for $H_g(t)$ in $\tau$:
\begin{eqnarray*}
H_{0}(z) & = &  t \; (1-t),\\
H_{1}(z) & = &  \frac{t^3}{(1+t)(1-2 t)^2},
\end{eqnarray*} 
\begin{eqnarray*}
H_{2}(z) & = &  \frac{4 \; t^5 \; 
(1+2 t) \; (t^4-t^3+6 \;
t^2+t+2)} {(1+t)^5 (1-2 t)^7},
\\
H_{3}(z) & = &  4 \; t^7 \; (1+2 t)
\; (1+t)^{-9} (1-2 t)^{-12} \; (80 \;
t^9-120 \; t^8+1500 \; t^7+1036 \; t^6
\\
& &
 + \ 3768 \;
t^5+2820 \; t^4 +2288 \; t^3 + 1008 \; t^2+258 \;
 t+45),
\\
H_{4}(z) & = &  4 \; t^9 \; (1+2 t)
\; (1+t)^{-13} (1-2 t)^{-17} \; (16768 \;
t^{14}-33536 \; t^{13}
 \\
 & & 
 + \ 653776 \; t^{12}+786480 \; t^{11}+ 4358016
 \; t^{10} + 6151056 \; t^9+10059552 \; t^8 \\
 & & 
 + \ 10217040 \; t^7 + 8418240 \;
 t^6+5227024 \; t^5+2365888 \; t^4
 + 800128
 \; t^3\\
 & & 
 +\ 181665 \; t^2 + 25992 \; t+2016),
\\
H_{5}(z) & = &  4 \; t^{11} \; (1+2 t)
\; (1+t)^{-17} (1-2 t)^{-22} \;  (6732800
\; t^{19}-16832000 \; t^{18}
\\
& &
 +\ 450011520 \; t^{17}+ 773106240 \;
 t^{16}+5764983552 \; t^{15}+11910647232 \; t^{14}
 \\
 & &
  +\ 29130502912 \; t^{13}+46090300928 \;
 t^{12}+63452543616 \; t^{11}
\\
& &
 +\ 68713116608 \;
 t^{10}+60654218080 \; t^9+43591208976 \; t^8
 \\
 & & +\ 25142796864 \; t^7+11637842232 \;
 t^6+4232899206 \; t^5+1181820745 \; t^4
\\
& &
 +\ 245635580 \; t^3+35501760 \; t^2+3255120
 \; t+151200), 
\end{eqnarray*}
\begin{eqnarray*}
H_{6}(z)  & = & 4 \; t^{13} \; (1+2t) 
\; (1+t)^{-21}  (1-2t)^{-27}  \; (4424052736
\; t^{24}-13272158208 \; t^{23} \\
& &
 +\ 452750478336 \; t^{22} + 1012254206976
\; t^{21} +9488911137792 \; t^{20}
\\
& & 
 +\ 25803592571904
\; t^{19} + 83891900050944 \; t^{18} +180120643165440
\; t^{17} \\
& &
+\ 346626234587904 \; t^{16}  + 535272874975232
\; t^{15} +701152993531392 \; t^{14}  \\
& &
+ \ 771688966862592
\; t^{13} + 716686355273472 \; t^{12} +563018634260736
\; t^{11} \\
& & 
+\ 372549313187520 \; t^{10} + 207088794784752 \;
t^9+96021082581732 \; t^8 \\
& &
+ \ 36765061031004 \; t^7 + 11475757049569 \;
t^6+2863185376896 \; t^5\\
&&
+\ 556090776432 \;
t^4+80913152016 \; t^3+ 8274846384 \; t^2+536428224
\; t\\
&& +\ 17107200).
\end{eqnarray*}

For $0 \leq g \leq 3$, these expressions correspond to $F_g(t)$ in Zograf's
communication. Moreover, they reveal an extra factorization 
by $4 (1+2 t)$ for $g \geq 2$.

\section{Efficient enumeration of rooted and sensed unrooted hypermaps by number
of darts, vertices and hyperedges}
\label{unrooted:hyp:sec}

We recall that a sensed map or hypermap is an equivalence class of (unrooted)
maps or hypermaps under orientation-preserving isomorphism.

Before enumerating sensed hypermaps we first need to enumerate rooted hypermaps.
 We use an efficient method of counting rooted hypermaps by number of darts,
faces, vertices and hyperedges or, equivalently~\cite{Walsh75}, 2-coloured
bipartite maps rooted at a white vertex by number of edges, faces, white vertices and black
vertices, presented by Kazarian and Zograf~\cite{KZ15}, and then count sensed
2-coloured bipartite maps and hypermaps with the same parameters using the same
method we used~\cite{WGM12,GW14} to count sensed maps by number of edges, faces
and vertices. The recurrence (formula (11) in \cite{KZ15}), with $f$ changed to
$H$, is as follows. Define $H_{g,d}$ to be a
homogeneous polynomial in the three variables $t$, $u$, and $v$.  The coefficient of $t^f u^b v^w$ in
$H_{g,d}$ is the number of 2-coloured bipartite maps of genus $g$ with $d$
edges, $f$ faces, $b$ black vertices and $w$ white vertices rooted at a white
vertex or, equivalently, the number of rooted hypermaps of genus $g$ with $d$
darts, $f$ faces, $b$ hyperedges and $w$ vertices.  Then $H_{0,1} = tuv$ and
\begin{eqnarray}
& & (d+1) H_{g,d} =
\nonumber
\\
& & \quad (2d-1)(t+u+v)H_{g,d-1} \nonumber
\\
& & \quad + \ (d-2) \left(2(tu+tv+uv)-(t^2+u^2+v^2)\right) H_{g,d-2}
\label{unrooted:1:eq}
\\
& & \quad + \ (d-1)^2 (d-2) H_{g-1,d-2}  + \sum_{i=0}^{g} \sum_{j=1}^{d-3}
(4+6j)(d-2-j)H_{i,j}H_{g-i,d-2-j}.
\nonumber
\end{eqnarray}

In~\cite{WGM12} we collaborated with Mednykh to enumerate rooted and sensed
maps. Mednykh enumerated maps of genus up to 11 by number of edges alone, while
we enumerated maps of genus up to 10 by number of edges and vertices.  The
method we used to enumerate rooted maps is presented in~\cite{WG14}.  The method
we used to enumerate sensed maps is based on Liskovets'
refinement~\cite{Liskovets10} of the method Mednykh and Nedela used to enumerate
sensed map of genus up to 3 by number of edges~\cite{MN06}.
Later we used a more efficient method of enumerating rooted maps, presented
in~\cite{CC15}, to enumerate rooted and sensed maps of genus up to
50~\cite{GW14}.

To describe here the modifications we made to pass from maps to 2-coloured
bipartite maps we need to briefly discuss a few of the concepts described in
more detail in~\cite{WGM12}.  All the automorphisms of a map on an orientable
surface are periodic.  If the period is ${L} > 1$, then the automorphism divides
the map into ${L}$ isomorphic copies of a smaller map, called the \emph{quotient
map}.  Most of the \emph{cells} (vertices, edges and faces) are in orbits of
length ${L}$ under the automorphism; those that aren't are called \emph{branch
points}. For example, if a map is drawn on the surface of a sphere which
undergoes a rotation through $360/{L}$ degrees, the two cells through which the
axis of rotation pass are fixed; so they are each in an orbit of length $1$ for
any ${L}$.  For maps of higher genus, not all the branch points are on orbits of
length $1$.  For example, if a torus is represented as a square with opposite
edges identified in pairs, and is rotated by $90$ degrees (period $4$), then the
centre of the square is a branch point of orbit length $1$ and so is the point
represented by all four corners of the square, but the middle of the sides of
the square are two branch points of orbit length $2$: the point represented by
the middle of both vertical sides of the square is taken by the rotation onto
the point represented by the middle of both horizontal sides, and vice versa; so
it takes two rotations to take either of these points back onto itself.  Also,
if the middle of an edge is a branch point, then the quotient map contains half
of that edge – a \emph{dangling semi-edge}.

An automorphism of a map $M$ of genus $G$ is characterized by the following
parameters: the period ${L}$, the genus $g$ of its quotient map and the number of
branch points of each orbit length.  If each orbit length is replaced by its
\emph{branch index} (${L}$ divided by the orbit length), we obtain what is called
an \emph{orbifold signature} in~\cite{MN06}.  In~\cite{MN06} a method is
presented for determining which orbifold signatures could characterize an
automorphism of a map of genus $G$ (a $G$-\emph{admissible orbifold}) and how
many such automorphisms could be characterized by that orbifold signature; a
variant of that method is presented in~\cite{Liskovets10}, and this is the one
we use except that we deal with orbit lengths instead of branch indices.  The
method used in~\cite{MN06} to enumerate sensed maps of genus $G$ with $E$ edges
by number of edges can be roughly described as follows.  For each $G$-admissible
orbifold $O$, let $g$ be the genus of the quotient map, ${L}$ be the period and
$q_i$ be the number of branch points with branch index $i$.  Then the number
$\nu_O(d)$ of rooted maps with $d$ darts that could serve as a quotient map for
an automorphism with that signature once the branch points are pasted onto the map
in all possible ways is given by
\begin{equation}
\nu_O(d) = \sum_{s=0}^{q_2}{d\choose s}
{(d-s) /2\!+\! 2\!-\!2g\choose
q_2\!-\!s,q_3,\dots,q_{L}}{N}_g((d-s)/2),
\label{nuO:eq}
\end{equation}
where $N_g(n)$ is the number of rooted maps of genus $g$ with $n$ edges ($0$ if
$n$ is not an integer).  Here $s$ is the number of dangling semi-edges in the
quotient map $m$, all of which must be in orbits of length ${L}/2$ so that they
represent normal edges in the original map $M$.  The binomial coefficient is the
number of ways of inserting dangling semi-edges into the rooted map multiplied
by $d/(d-s)$ because there are $d$ ways to root the map once the dangling edges
have been inserted and only $d-s$ ways to root it without the dangling edges.
The multinomial coefficient is the number of ways to distribute the branch
points with the various branch indices among the non-edges of the quotient map;
the number at the top of the multinomial coefficient is the number of non-edges
and is given by the Euler-Poincaré formula~(\ref{euler:eq}).
Then the number of sensed maps of genus $G$ with $E$ edges
is
\begin{equation}
\frac{1}{2E} \sum_{{L} \mid E}
\sum_{O} Epi_0(\pi_1(O),Z_{{L}})\, \nu_O(2E / {L}),
\label{unrooted:eq}
\end{equation}
where $O$ runs over all the $G$-admissible orbifolds with period ${L}$ and
$Epi_0(\pi_1(O),Z_{{L}})$ is the number of automorphisms that have the orbifold
signature of $O$.

In~\cite{WGM12} we distributed the branch points that aren't on dangling
semi-edges among the vertices and faces separately.  The quotient map of a
bipartite map can't contain any dangling semi-edges; otherwise the lifted map
would have an edge joining two vertices of the same colour. Here we distribute
the branch points among the white vertices, black vertices and faces, and, like
in~\cite{WGM12}, we don't use a formula like (\ref{unrooted:eq});  instead we
compute the contribution of each orbifold signature to the number of sensed
$2$-coloured bipartite maps whose number of white vertices, black vertices,
faces and edges are allowed to vary within a user-defined upper bound on the
number of edges.

Suppose that the quotient map is of genus $g$ and has $w$ white vertices, $b$
black vertices and $f$ faces.  Then the number $e$ of edges can be
calculated from the formula
\begin{equation}
f - e + w + b = 2(1-g)  \label{hyper:genus:eq}                                                        
\end{equation}
and the number $d$ of darts is $2e$.  Suppose also that among the branch
points of orbit length $i$, $w_i$ are on a white vertex, $b_i$ are on a black
vertex and $f_i$ are in a face.  We denote by $w_L$, $b_L$ and $f_L$ the number
of white vertices, black vertices and faces, respectively, that do not contain a
branch point.  The original map will have $W$ white vertices, $B$ black vertices
and $F$ faces, where

\begin{equation}
W = \sum_{i=1}^{{L}} i w_i, B = \sum_{i=1}^{{L}} i b_i \text{ and } F =
\sum_{i=1}^{{L}} i f_i, \label{WBF:eq}
\end{equation}
and the total number $E$ of edges is equal to ${L} \, e = F + W + B –
2(1-g)$.

The binomial coefficient in (\ref{nuO:eq}) disappears because the quotient map
can't contain any dangling semi-edges. The multinomial coefficient must be
replaced by the number of ways to distribute the branch points among the white
vertices, black vertices and faces.  Then (\ref{nuO:eq}) becomes

\begin{eqnarray}
& & \nu_O(d,w,b,f) = \nonumber \\
& & \quad {w\choose w_1,w_2,\ldots,w_{{L}}}
{b\choose b_1,b_2,\ldots,b_{{L}}}
{f\choose f_1,f_2,\ldots,f_{{L}}}
 N_g(d,w,b,f),
\label{nuO:hyp:eq}
\end{eqnarray}
where $d$ is the number of edges in the quotient maps on both sides of the
formula (or the number of darts in the corresponding hypermaps)
and $N_g(d,w,b,f)$ is the number of $2$-coloured bipartite maps with $d$ edges
with $w$ white vertices, $b$ black vertices and $f$ faces, rooted at a white
vertex.  For this number to be positive, the sum of all the $w_i$ cannot exceed
$w$ with a similar bound on the sum of all the $b_i$ and the sum of all the
$f_i$; so $w$, $b$ and $f$ each starts at its respective sum and increases by
$1$ until the number $E$ of edges in the original map exceeds a user-defined
maximum.  With each increase of $w$, $b$ or $f$, one of the multinomial
coefficients in (\ref{nuO:hyp:eq}) gets updated using a single multiplication
and division. The product of these three multinomial coefficients must be
computed for all sets of non-negative integers such that for each $i$,
$w_i+b_i+f_i$ is equal to the total number of branch points of orbit length $i$.

Once (\ref{nuO:hyp:eq}) is multiplied by the number of automorphisms with the
current orbifold signature, we get the contribution of that signature and the
numbers $w_i$, $b_i$ and $f_i$ to $E$ times the number of sensed $2$-coloured
bipartite maps of genus $G$ with $E$ edges, $F$ faces, $B$ black vertices and
$W$ white vertices.  This contribution is added to the appropriate element of an
array, initially $0$, and when all the contributions have been tallied, for each
$E$, $F$, $W$ and $B$ the corresponding array element is divided by $E$ (not
$2E$ because the root must be incident to a white vertex) to give the number of
sensed $2$-coloured bipartite maps of genus $G$ with $E$ edges, $F$ faces, $B$
black vertices and $W$ white vertices or, equivalently, the number of sensed
hypermaps of genus $G$ with $E$ darts, $F$ faces, $B$ hyperedges and $W$
vertices.

This enumeration was done with a program written in C++ using CLN to treat big
integers.  It enumerated rooted and sensed hypermaps of genus up to 24 with up
to 50 darts as fast as it could display the numbers on the screen.  The numbers
coincide with those obtained by generating the hypermaps~\cite{Walsh15}. 
The source code is available from the second author on request.

\bibliographystyle{plain}

\newpage

\appendix

\section{First numbers of rooted hypermaps}
\label{rooted:appendix}

The following sections show the numbers \texttt{h}
of rooted hypermaps of genus $g$ with $d$ darts, $v$ vertices, $e$ edges and $d
- v - e + 2(1 – g)$ faces, for $g \leq 6$ and $d \leq 14$.

\subsection{Genus 0}
\begin{multicols}{3}
\begin{scriptsize}
\begin{verbatim}
   d   v   e   f   h
   1   1   1   1   1

   1         sum   1

   2   1   1   2   1
   2   1   2   1   1
   2   2   1   1   1

   2         sum   3

   3   1   1   3   1
   3   1   2   2   3
   3   2   1   2   3
   3   1   3   1   1
   3   2   2   1   3
   3   3   1   1   1

   3         sum   12

   4   1   1   4   1
   4   1   2   3   6
   4   2   1   3   6
   4   1   3   2   6
   4   2   2   2   17
   4   3   1   2   6
   4   1   4   1   1
   4   2   3   1   6
   4   3   2   1   6
   4   4   1   1   1

   4         sum   56

   5   1   1   5   1
   5   1   2   4   10
   5   2   1   4   10
   5   1   3   3   20
   5   2   2   3   55
   5   3   1   3   20
   5   1   4   2   10
   5   2   3   2   55
   5   3   2   2   55
   5   4   1   2   10
   5   1   5   1   1
   5   2   4   1   10
   5   3   3   1   20
   5   4   2   1   10
   5   5   1   1   1

   5         sum   288

   6   1   1   6   1
   6   1   2   5   15
   6   2   1   5   15
   6   1   3   4   50
   6   2   2   4   135
   6   3   1   4   50
   6   1   4   3   50
   6   2   3   3   262
   6   3   2   3   262
   6   4   1   3   50
   6   1   5   2   15
   6   2   4   2   135
   6   3   3   2   262
   6   4   2   2   135
   6   5   1   2   15
   6   1   6   1   1
   6   2   5   1   15
   6   3   4   1   50
   6   4   3   1   50
   6   5   2   1   15
   6   6   1   1   1

   6         sum   1584

   7   1   1   7   1
   7   1   2   6   21
   7   2   1   6   21
   7   1   3   5   105
   7   2   2   5   280
   7   3   1   5   105
   7   1   4   4   175
   7   2   3   4   889
   7   3   2   4   889
   7   4   1   4   175
   7   1   5   3   105
   7   2   4   3   889
   7   3   3   3   1694
   7   4   2   3   889
   7   5   1   3   105
   7   1   6   2   21
   7   2   5   2   280
   7   3   4   2   889
   7   4   3   2   889
   7   5   2   2   280
   7   6   1   2   21
   7   1   7   1   1
   7   2   6   1   21
   7   3   5   1   105
   7   4   4   1   175
   7   5   3   1   105
   7   6   2   1   21
   7   7   1   1   1

   7         sum   9152

   8   1   1   8   1
   8   1   2   7   28
   8   2   1   7   28
   8   1   3   6   196
   8   2   2   6   518
   8   3   1   6   196
   8   1   4   5   490
   8   2   3   5   2436
   8   3   2   5   2436
   8   4   1   5   490
   8   1   5   4   490
   8   2   4   4   3985
   8   3   3   4   7500
   8   4   2   4   3985
   8   5   1   4   490
   8   1   6   3   196
   8   2   5   3   2436
   8   3   4   3   7500
   8   4   3   3   7500
   8   5   2   3   2436
   8   6   1   3   196
   8   1   7   2   28
   8   2   6   2   518
   8   3   5   2   2436
   8   4   4   2   3985
   8   5   3   2   2436
   8   6   2   2   518
   8   7   1   2   28
   8   1   8   1   1
   8   2   7   1   28
   8   3   6   1   196
   8   4   5   1   490
   8   5   4   1   490
   8   6   3   1   196
   8   7   2   1   28
   8   8   1   1   1

   8         sum   54912

   9   1   1   9   1
   9   1   2   8   36
   9   2   1   8   36
   9   1   3   7   336
   9   2   2   7   882
   9   3   1   7   336
   9   1   4   6   1176
   9   2   3   6   5754
   9   3   2   6   5754
   9   4   1   6   1176
   9   1   5   5   1764
   9   2   4   5   13941
   9   3   3   5   26004
   9   4   2   5   13941
   9   5   1   5   1764
   9   1   6   4   1176
   9   2   5   4   13941
   9   3   4   4   42015
   9   4   3   4   42015
   9   5   2   4   13941
   9   6   1   4   1176
   9   1   7   3   336
   9   2   6   3   5754
   9   3   5   3   26004
   9   4   4   3   42015
   9   5   3   3   26004
   9   6   2   3   5754
   9   7   1   3   336
   9   1   8   2   36
   9   2   7   2   882
   9   3   6   2   5754
   9   4   5   2   13941
   9   5   4   2   13941
   9   6   3   2   5754
   9   7   2   2   882
   9   8   1   2   36
   9   1   9   1   1
   9   2   8   1   36
   9   3   7   1   336
   9   4   6   1   1176
   9   5   5   1   1764
   9   6   4   1   1176
   9   7   3   1   336
   9   8   2   1   36
   9   9   1   1   1

   9         sum   339456

  10   1   1  10   1
  10   1   2   9   45
  10   2   1   9   45
  10   1   3   8   540
  10   2   2   8   1410
  10   3   1   8   540
  10   1   4   7   2520
  10   2   3   7   12180
  10   3   2   7   12180
  10   4   1   7   2520
  10   1   5   6   5292
  10   2   4   6   40935
  10   3   3   6   75840
  10   4   2   6   40935
  10   5   1   6   5292
  10   1   6   5   5292
  10   2   5   5   60626
  10   3   4   5   179860
  10   4   3   5   179860
  10   5   2   5   60626
  10   6   1   5   5292
  10   1   7   4   2520
  10   2   6   4   40935
  10   3   5   4   179860
  10   4   4   4   288025
  10   5   3   4   179860
  10   6   2   4   40935
  10   7   1   4   2520
  10   1   8   3   540
  10   2   7   3   12180
  10   3   6   3   75840
  10   4   5   3   179860
  10   5   4   3   179860
  10   6   3   3   75840
  10   7   2   3   12180
  10   8   1   3   540
  10   1   9   2   45
  10   2   8   2   1410
  10   3   7   2   12180
  10   4   6   2   40935
  10   5   5   2   60626
  10   6   4   2   40935
  10   7   3   2   12180
  10   8   2   2   1410
  10   9   1   2   45
  10   1  10   1   1
  10   2   9   1   45
  10   3   8   1   540
  10   4   7   1   2520
  10   5   6   1   5292
  10   6   5   1   5292
  10   7   4   1   2520
  10   8   3   1   540
  10   9   2   1   45
  10  10   1   1   1

  10         sum   2149888

  11   1   1  11   1
  11   1   2  10   55
  11   2   1  10   55
  11   1   3   9   825
  11   2   2   9   2145
  11   3   1   9   825
  11   1   4   8   4950
  11   2   3   8   23694
  11   3   2   8   23694
  11   4   1   8   4950
  11   1   5   7   13860
  11   2   4   7   105435
  11   3   3   7   194304
  11   4   2   7   105435
  11   5   1   7   13860
  11   1   6   6   19404
  11   2   5   6   216601
  11   3   4   6   634865
  11   4   3   6   634865
  11   5   2   6   216601
  11   6   1   6   19404
  11   1   7   5   13860
  11   2   6   5   216601
  11   3   5   5   931854
  11   4   4   5   1482250
  11   5   3   5   931854
  11   6   2   5   216601
  11   7   1   5   13860
  11   1   8   4   4950
  11   2   7   4   105435
  11   3   6   4   634865
  11   4   5   4   1482250
  11   5   4   4   1482250
  11   6   3   4   634865
  11   7   2   4   105435
  11   8   1   4   4950
  11   1   9   3   825
  11   2   8   3   23694
  11   3   7   3   194304
  11   4   6   3   634865
  11   5   5   3   931854
  11   6   4   3   634865
  11   7   3   3   194304
  11   8   2   3   23694
  11   9   1   3   825
  11   1  10   2   55
  11   2   9   2   2145
  11   3   8   2   23694
  11   4   7   2   105435
  11   5   6   2   216601
  11   6   5   2   216601
  11   7   4   2   105435
  11   8   3   2   23694
  11   9   2   2   2145
  11  10   1   2   55
  11   1  11   1   1
  11   2  10   1   55
  11   3   9   1   825
  11   4   8   1   4950
  11   5   7   1   13860
  11   6   6   1   19404
  11   7   5   1   13860
  11   8   4   1   4950
  11   9   3   1   825
  11  10   2   1   55
  11  11   1   1   1

  11         sum   13891584

  12   1   1  12   1
  12   1   2  11   66
  12   2   1  11   66
  12   1   3  10   1210
  12   2   2  10   3135
  12   3   1  10   1210
  12   1   4   9   9075
  12   2   3   9   43098
  12   3   2   9   43098
  12   4   1   9   9075
  12   1   5   8   32670
  12   2   4   8   245223
  12   3   3   8   449988
  12   4   2   8   245223
  12   5   1   8   32670
  12   1   6   7   60984
  12   2   5   7   666996
  12   3   4   7   1936308
  12   4   3   7   1936308
  12   5   2   7   666996
  12   6   1   7   60984
  12   1   7   6   60984
  12   2   6   6   925190
  12   3   5   6   3915576
  12   4   4   6   6195560
  12   5   3   6   3915576
  12   6   2   6   925190
  12   7   1   6   60984
  12   1   8   5   32670
  12   2   7   5   666996
  12   3   6   5   3915576
  12   4   5   5   9032898
  12   5   4   5   9032898
  12   6   3   5   3915576
  12   7   2   5   666996
  12   8   1   5   32670
  12   1   9   4   9075
  12   2   8   4   245223
  12   3   7   4   1936308
  12   4   6   4   6195560
  12   5   5   4   9032898
  12   6   4   4   6195560
  12   7   3   4   1936308
  12   8   2   4   245223
  12   9   1   4   9075
  12   1  10   3   1210
  12   2   9   3   43098
  12   3   8   3   449988
  12   4   7   3   1936308
  12   5   6   3   3915576
  12   6   5   3   3915576
  12   7   4   3   1936308
  12   8   3   3   449988
  12   9   2   3   43098
  12  10   1   3   1210
  12   1  11   2   66
  12   2  10   2   3135
  12   3   9   2   43098
  12   4   8   2   245223
  12   5   7   2   666996
  12   6   6   2   925190
  12   7   5   2   666996
  12   8   4   2   245223
  12   9   3   2   43098
  12  10   2   2   3135
  12  11   1   2   66
  12   1  12   1   1
  12   2  11   1   66
  12   3  10   1   1210
  12   4   9   1   9075
  12   5   8   1   32670
  12   6   7   1   60984
  12   7   6   1   60984
  12   8   5   1   32670
  12   9   4   1   9075
  12  10   3   1   1210
  12  11   2   1   66
  12  12   1   1   1

  12         sum   91287552

  13   1   1  13   1
  13   1   2  12   78
  13   2   1  12   78
  13   1   3  11   1716
  13   2   2  11   4433
  13   3   1  11   1716
  13   1   4  10   15730
  13   2   3  10   74217
  13   3   2  10   74217
  13   4   1  10   15730
  13   1   5   9   70785
  13   2   4   9   525525
  13   3   3   9   960960
  13   4   2   9   525525
  13   5   1   9   70785
  13   1   6   8   169884
  13   2   5   8   1827683
  13   3   4   8   5264545
  13   4   3   8   5264545
  13   5   2   8   1827683
  13   6   1   8   169884
  13   1   7   7   226512
  13   2   6   7   3356522
  13   3   5   7   14019928
  13   4   4   7   22089600
  13   5   3   7   14019928
  13   6   2   7   3356522
  13   7   1   7   226512
  13   1   8   6   169884
  13   2   7   6   3356522
  13   3   6   6   19315114
  13   4   5   6   44136820
  13   5   4   6   44136820
  13   6   3   6   19315114
  13   7   2   6   3356522
  13   8   1   6   169884
  13   1   9   5   70785
  13   2   8   5   1827683
  13   3   7   5   14019928
  13   4   6   5   44136820
  13   5   5   5   64013222
  13   6   4   5   44136820
  13   7   3   5   14019928
  13   8   2   5   1827683
  13   9   1   5   70785
  13   1  10   4   15730
  13   2   9   4   525525
  13   3   8   4   5264545
  13   4   7   4   22089600
  13   5   6   4   44136820
  13   6   5   4   44136820
  13   7   4   4   22089600
  13   8   3   4   5264545
  13   9   2   4   525525
  13  10   1   4   15730
  13   1  11   3   1716
  13   2  10   3   74217
  13   3   9   3   960960
  13   4   8   3   5264545
  13   5   7   3   14019928
  13   6   6   3   19315114
  13   7   5   3   14019928
  13   8   4   3   5264545
  13   9   3   3   960960
  13  10   2   3   74217
  13  11   1   3   1716
  13   1  12   2   78
  13   2  11   2   4433
  13   3  10   2   74217
  13   4   9   2   525525
  13   5   8   2   1827683
  13   6   7   2   3356522
  13   7   6   2   3356522
  13   8   5   2   1827683
  13   9   4   2   525525
  13  10   3   2   74217
  13  11   2   2   4433
  13  12   1   2   78
  13   1  13   1   1
  13   2  12   1   78
  13   3  11   1   1716
  13   4  10   1   15730
  13   5   9   1   70785
  13   6   8   1   169884
  13   7   7   1   226512
  13   8   6   1   169884
  13   9   5   1   70785
  13  10   4   1   15730
  13  11   3   1   1716
  13  12   2   1   78
  13  13   1   1   1

  13         sum   608583680

  14   1   1  14   1
  14   1   2  13   91
  14   2   1  13   91
  14   1   3  12   2366
  14   2   2  12   6097
  14   3   1  12   2366
  14   1   4  11   26026
  14   2   3  11   122122
  14   3   2  11   122122
  14   4   1  11   26026
  14   1   5  10   143143
  14   2   4  10   1053052
  14   3   3  10   1919918
  14   4   2  10   1053052
  14   5   1  10   143143
  14   1   6   9   429429
  14   2   5   9   4557553
  14   3   4   9   13043030
  14   4   3   9   13043030
  14   5   2   9   4557553
  14   6   1   9   429429
  14   1   7   8   736164
  14   2   6   8   10701873
  14   3   5   8   44221632
  14   4   4   8   69432090
  14   5   3   8   44221632
  14   6   2   8   10701873
  14   7   1   8   736164
  14   1   8   7   736164
  14   2   7   7   14168988
  14   3   6   7   80231508
  14   4   5   7   181925268
  14   5   4   7   181925268
  14   6   3   7   80231508
  14   7   2   7   14168988
  14   8   1   7   736164
  14   1   9   6   429429
  14   2   8   6   10701873
  14   3   7   6   80231508
  14   4   6   6   249321114
  14   5   5   6   360078558
  14   6   4   6   249321114
  14   7   3   6   80231508
  14   8   2   6   10701873
  14   9   1   6   429429
  14   1  10   5   143143
  14   2   9   5   4557553
  14   3   8   5   44221632
  14   4   7   5   181925268
  14   5   6   5   360078558
  14   6   5   5   360078558
  14   7   4   5   181925268
  14   8   3   5   44221632
  14   9   2   5   4557553
  14  10   1   5   143143
  14   1  11   4   26026
  14   2  10   4   1053052
  14   3   9   4   13043030
  14   4   8   4   69432090
  14   5   7   4   181925268
  14   6   6   4   249321114
  14   7   5   4   181925268
  14   8   4   4   69432090
  14   9   3   4   13043030
  14  10   2   4   1053052
  14  11   1   4   26026
  14   1  12   3   2366
  14   2  11   3   122122
  14   3  10   3   1919918
  14   4   9   3   13043030
  14   5   8   3   44221632
  14   6   7   3   80231508
  14   7   6   3   80231508
  14   8   5   3   44221632
  14   9   4   3   13043030
  14  10   3   3   1919918
  14  11   2   3   122122
  14  12   1   3   2366
  14   1  13   2   91
  14   2  12   2   6097
  14   3  11   2   122122
  14   4  10   2   1053052
  14   5   9   2   4557553
  14   6   8   2   10701873
  14   7   7   2   14168988
  14   8   6   2   10701873
  14   9   5   2   4557553
  14  10   4   2   1053052
  14  11   3   2   122122
  14  12   2   2   6097
  14  13   1   2   91
  14   1  14   1   1
  14   2  13   1   91
  14   3  12   1   2366
  14   4  11   1   26026
  14   5  10   1   143143
  14   6   9   1   429429
  14   7   8   1   736164
  14   8   7   1   736164
  14   9   6   1   429429
  14  10   5   1   143143
  14  11   4   1   26026
  14  12   3   1   2366
  14  13   2   1   91
  14  14   1   1   1

  14         sum   4107939840

\end{verbatim}
\end{scriptsize}
\end{multicols}

\newpage
\subsection{Genus 1}
\begin{multicols}{3}
\begin{scriptsize}
\begin{verbatim}
   d   v   e   f   h
   3   1   1   1   1

   3         sum   1

   4   1   1   2   5
   4   1   2   1   5
   4   2   1   1   5

   4         sum   15

   5   1   1   3   15
   5   1   2   2   40
   5   2   1   2   40
   5   1   3   1   15
   5   2   2   1   40
   5   3   1   1   15

   5         sum   165

   6   1   1   4   35
   6   1   2   3   175
   6   2   1   3   175
   6   1   3   2   175
   6   2   2   2   456
   6   3   1   2   175
   6   1   4   1   35
   6   2   3   1   175
   6   3   2   1   175
   6   4   1   1   35

   6         sum   1611

   7   1   1   5   70
   7   1   2   4   560
   7   2   1   4   560
   7   1   3   3   1050
   7   2   2   3   2695
   7   3   1   3   1050
   7   1   4   2   560
   7   2   3   2   2695
   7   3   2   2   2695
   7   4   1   2   560
   7   1   5   1   70
   7   2   4   1   560
   7   3   3   1   1050
   7   4   2   1   560
   7   5   1   1   70

   7         sum   14805

   8   1   1   6   126
   8   1   2   5   1470
   8   2   1   5   1470
   8   1   3   4   4410
   8   2   2   4   11199
   8   3   1   4   4410
   8   1   4   3   4410
   8   2   3   3   20684
   8   3   2   3   20684
   8   4   1   3   4410
   8   1   5   2   1470
   8   2   4   2   11199
   8   3   3   2   20684
   8   4   2   2   11199
   8   5   1   2   1470
   8   1   6   1   126
   8   2   5   1   1470
   8   3   4   1   4410
   8   4   3   1   4410
   8   5   2   1   1470
   8   6   1   1   126

   8         sum   131307

   9   1   1   7   210
   9   1   2   6   3360
   9   2   1   6   3360
   9   1   3   5   14700
   9   2   2   5   37035
   9   3   1   5   14700
   9   1   4   4   23520
   9   2   3   4   108285
   9   3   2   4   108285
   9   4   1   4   23520
   9   1   5   3   14700
   9   2   4   3   108285
   9   3   3   3   197896
   9   4   2   3   108285
   9   5   1   3   14700
   9   1   6   2   3360
   9   2   5   2   37035
   9   3   4   2   108285
   9   4   3   2   108285
   9   5   2   2   37035
   9   6   1   2   3360
   9   1   7   1   210
   9   2   6   1   3360
   9   3   5   1   14700
   9   4   4   1   23520
   9   5   3   1   14700
   9   6   2   1   3360
   9   7   1   1   210

   9         sum   1138261

  10   1   1   8   330
  10   1   2   7   6930
  10   2   1   7   6930
  10   1   3   6   41580
  10   2   2   6   104115
  10   3   1   6   41580
  10   1   4   5   97020
  10   2   3   5   440440
  10   3   2   5   440440
  10   4   1   5   97020
  10   1   5   4   97020
  10   2   4   4   697250
  10   3   3   4   1264310
  10   4   2   4   697250
  10   5   1   4   97020
  10   1   6   3   41580
  10   2   5   3   440440
  10   3   4   3   1264310
  10   4   3   3   1264310
  10   5   2   3   440440
  10   6   1   3   41580
  10   1   7   2   6930
  10   2   6   2   104115
  10   3   5   2   440440
  10   4   4   2   697250
  10   5   3   2   440440
  10   6   2   2   104115
  10   7   1   2   6930
  10   1   8   1   330
  10   2   7   1   6930
  10   3   6   1   41580
  10   4   5   1   97020
  10   5   4   1   97020
  10   6   3   1   41580
  10   7   2   1   6930
  10   8   1   1   330

  10         sum   9713835

  11   1   1   9   495
  11   1   2   8   13200
  11   2   1   8   13200
  11   1   3   7   103950
  11   2   2   7   259017
  11   3   1   7   103950
  11   1   4   6   332640
  11   2   3   6   1493525
  11   3   2   6   1493525
  11   4   1   6   332640
  11   1   5   5   485100
  11   2   4   5   3420835
  11   3   3   5   6165478
  11   4   2   5   3420835
  11   5   1   5   485100
  11   1   6   4   332640
  11   2   5   4   3420835
  11   3   4   4   9684433
  11   4   3   4   9684433
  11   5   2   4   3420835
  11   6   1   4   332640
  11   1   7   3   103950
  11   2   6   3   1493525
  11   3   5   3   6165478
  11   4   4   3   9684433
  11   5   3   3   6165478
  11   6   2   3   1493525
  11   7   1   3   103950
  11   1   8   2   13200
  11   2   7   2   259017
  11   3   6   2   1493525
  11   4   5   2   3420835
  11   5   4   2   3420835
  11   6   3   2   1493525
  11   7   2   2   259017
  11   8   1   2   13200
  11   1   9   1   495
  11   2   8   1   13200
  11   3   7   1   103950
  11   4   6   1   332640
  11   5   5   1   485100
  11   6   4   1   332640
  11   7   3   1   103950
  11   8   2   1   13200
  11   9   1   1   495

  11         sum   81968469

  12   1   1  10   715
  12   1   2   9   23595
  12   2   1   9   23595
  12   1   3   8   235950
  12   2   2   8   585585
  12   3   1   8   235950
  12   1   4   7   990990
  12   2   3   7   4410120
  12   3   2   7   4410120
  12   4   1   7   990990
  12   1   5   6   1981980
  12   2   4   6   13768300
  12   3   3   6   24695580
  12   4   2   6   13768300
  12   5   1   6   1981980
  12   1   6   5   1981980
  12   2   5   5   19920390
  12   3   4   5   55785870
  12   4   3   5   55785870
  12   5   2   5   19920390
  12   6   1   5   1981980
  12   1   7   4   990990
  12   2   6   4   13768300
  12   3   5   4   55785870
  12   4   4   4   87100531
  12   5   3   4   55785870
  12   6   2   4   13768300
  12   7   1   4   990990
  12   1   8   3   235950
  12   2   7   3   4410120
  12   3   6   3   24695580
  12   4   5   3   55785870
  12   5   4   3   55785870
  12   6   3   3   24695580
  12   7   2   3   4410120
  12   8   1   3   235950
  12   1   9   2   23595
  12   2   8   2   585585
  12   3   7   2   4410120
  12   4   6   2   13768300
  12   5   5   2   19920390
  12   6   4   2   13768300
  12   7   3   2   4410120
  12   8   2   2   585585
  12   9   1   2   23595
  12   1  10   1   715
  12   2   9   1   23595
  12   3   8   1   235950
  12   4   7   1   990990
  12   5   6   1   1981980
  12   6   5   1   1981980
  12   7   4   1   990990
  12   8   3   1   235950
  12   9   2   1   23595
  12  10   1   1   715

  12         sum   685888171

  13   1   1  11   1001
  13   1   2  10   40040
  13   2   1  10   40040
  13   1   3   9   495495
  13   2   2   9   1225653
  13   3   1   9   495495
  13   1   4   8   2642640
  13   2   3   8   11674663
  13   3   2   8   11674663
  13   4   1   8   2642640
  13   1   5   7   6936930
  13   2   4   7   47604648
  13   3   3   7   85050784
  13   4   2   7   47604648
  13   5   1   7   6936930
  13   1   6   6   9513504
  13   2   5   6   93880696
  13   3   4   6   260619268
  13   4   3   6   260619268
  13   5   2   6   93880696
  13   6   1   6   9513504
  13   1   7   5   6936930
  13   2   6   5   93880696
  13   3   5   5   374805834
  13   4   4   5   582408775
  13   5   3   5   374805834
  13   6   2   5   93880696
  13   7   1   5   6936930
  13   1   8   4   2642640
  13   2   7   4   47604648
  13   3   6   4   260619268
  13   4   5   4   582408775
  13   5   4   4   582408775
  13   6   3   4   260619268
  13   7   2   4   47604648
  13   8   1   4   2642640
  13   1   9   3   495495
  13   2   8   3   11674663
  13   3   7   3   85050784
  13   4   6   3   260619268
  13   5   5   3   374805834
  13   6   4   3   260619268
  13   7   3   3   85050784
  13   8   2   3   11674663
  13   9   1   3   495495
  13   1  10   2   40040
  13   2   9   2   1225653
  13   3   8   2   11674663
  13   4   7   2   47604648
  13   5   6   2   93880696
  13   6   5   2   93880696
  13   7   4   2   47604648
  13   8   3   2   11674663
  13   9   2   2   1225653
  13  10   1   2   40040
  13   1  11   1   1001
  13   2  10   1   40040
  13   3   9   1   495495
  13   4   8   1   2642640
  13   5   7   1   6936930
  13   6   6   1   9513504
  13   7   5   1   6936930
  13   8   4   1   2642640
  13   9   3   1   495495
  13  10   2   1   40040
  13  11   1   1   1001

  13         sum   5702382933

  14   1   1  12   1365
  14   1   2  11   65065
  14   2   1  11   65065
  14   1   3  10   975975
  14   2   2  10   2407405
  14   3   1  10   975975
  14   1   4   9   6441435
  14   2   3   9   28283255
  14   3   2   9   28283255
  14   4   1   9   6441435
  14   1   5   8   21471450
  14   2   4   8   145864355
  14   3   3   8   259750218
  14   4   2   8   145864355
  14   5   1   8   21471450
  14   1   6   7   38648610
  14   2   5   7   375707570
  14   3   4   7   1035514340
  14   4   3   7   1035514340
  14   5   2   7   375707570
  14   6   1   7   38648610
  14   1   7   6   38648610
  14   2   6   6   512104880
  14   3   5   6   2020140430
  14   4   4   6   3126887407
  14   5   3   6   2020140430
  14   6   2   6   512104880
  14   7   1   6   38648610
  14   1   8   5   21471450
  14   2   7   5   375707570
  14   3   6   5   2020140430
  14   4   5   5   4475516612
  14   5   4   5   4475516612
  14   6   3   5   2020140430
  14   7   2   5   375707570
  14   8   1   5   21471450
  14   1   9   4   6441435
  14   2   8   4   145864355
  14   3   7   4   1035514340
  14   4   6   4   3126887407
  14   5   5   4   4475516612
  14   6   4   4   3126887407
  14   7   3   4   1035514340
  14   8   2   4   145864355
  14   9   1   4   6441435
  14   1  10   3   975975
  14   2   9   3   28283255
  14   3   8   3   259750218
  14   4   7   3   1035514340
  14   5   6   3   2020140430
  14   6   5   3   2020140430
  14   7   4   3   1035514340
  14   8   3   3   259750218
  14   9   2   3   28283255
  14  10   1   3   975975
  14   1  11   2   65065
  14   2  10   2   2407405
  14   3   9   2   28283255
  14   4   8   2   145864355
  14   5   7   2   375707570
  14   6   6   2   512104880
  14   7   5   2   375707570
  14   8   4   2   145864355
  14   9   3   2   28283255
  14  10   2   2   2407405
  14  11   1   2   65065
  14   1  12   1   1365
  14   2  11   1   65065
  14   3  10   1   975975
  14   4   9   1   6441435
  14   5   8   1   21471450
  14   6   7   1   38648610
  14   7   6   1   38648610
  14   8   5   1   21471450
  14   9   4   1   6441435
  14  10   3   1   975975
  14  11   2   1   65065
  14  12   1   1   1365

  14         sum   47168678571

\end{verbatim}
\end{scriptsize}
\end{multicols}

\newpage
\subsection{Genus 2}
\begin{multicols}{3}
\begin{scriptsize}
\begin{verbatim}
   d   v   e   f   h
   5   1   1   1   8

   5         sum   8

   6   1   1   2   84
   6   1   2   1   84
   6   2   1   1   84

   6         sum   252

   7   1   1   3   469
   7   1   2   2   1183
   7   2   1   2   1183
   7   1   3   1   469
   7   2   2   1   1183
   7   3   1   1   469

   7         sum   4956

   8   1   1   4   1869
   8   1   2   3   8526
   8   2   1   3   8526
   8   1   3   2   8526
   8   2   2   2   21229
   8   3   1   2   8526
   8   1   4   1   1869
   8   2   3   1   8526
   8   3   2   1   8526
   8   4   1   1   1869

   8         sum   77992

   9   1   1   5   5985
   9   1   2   4   42588
   9   2   1   4   42588
   9   1   3   3   77028
   9   2   2   3   189999
   9   3   1   3   77028
   9   1   4   2   42588
   9   2   3   2   189999
   9   3   2   2   189999
   9   4   1   2   42588
   9   1   5   1   5985
   9   2   4   1   42588
   9   3   3   1   77028
   9   4   2   1   42588
   9   5   1   1   5985

   9         sum   1074564

  10   1   1   6   16401
  10   1   2   5   167013
  10   2   1   5   167013
  10   1   3   4   471240
  10   2   2   4   1154095
  10   3   1   4   471240
  10   1   4   3   471240
  10   2   3   3   2068070
  10   3   2   3   2068070
  10   4   1   3   471240
  10   1   5   2   167013
  10   2   4   2   1154095
  10   3   3   2   2068070
  10   4   2   2   1154095
  10   5   1   2   167013
  10   1   6   1   16401
  10   2   5   1   167013
  10   3   4   1   471240
  10   4   3   1   471240
  10   5   2   1   167013
  10   6   1   1   16401

  10         sum   13545216

  11   1   1   7   39963
  11   1   2   6   550011
  11   2   1   6   550011
  11   1   3   5   2221065
  11   2   2   5   5409019
  11   3   1   5   2221065
  11   1   4   4   3465000
  11   2   3   4   15014846
  11   3   2   4   15014846
  11   4   1   4   3465000
  11   1   5   3   2221065
  11   2   4   3   15014846
  11   3   3   3   26717482
  11   4   2   3   15014846
  11   5   1   3   2221065
  11   1   6   2   550011
  11   2   5   2   5409019
  11   3   4   2   15014846
  11   4   3   2   15014846
  11   5   2   2   5409019
  11   6   1   2   550011
  11   1   7   1   39963
  11   2   6   1   550011
  11   3   5   1   2221065
  11   4   4   1   3465000
  11   5   3   1   2221065
  11   6   2   1   550011
  11   7   1   1   39963

  11         sum   160174960

  12   1   1   8   88803
  12   1   2   7   1585584
  12   2   1   7   1585584
  12   1   3   6   8654646
  12   2   2   6   20981337
  12   3   1   6   8654646
  12   1   4   5   19324305
  12   2   3   5   82897296
  12   3   2   5   82897296
  12   4   1   5   19324305
  12   1   5   4   19324305
  12   2   4   4   128420004
  12   3   3   4   227256510
  12   4   2   4   128420004
  12   5   1   4   19324305
  12   1   6   3   8654646
  12   2   5   3   82897296
  12   3   4   3   227256510
  12   4   3   3   227256510
  12   5   2   3   82897296
  12   6   1   3   8654646
  12   1   7   2   1585584
  12   2   6   2   20981337
  12   3   5   2   82897296
  12   4   4   2   128420004
  12   5   3   2   82897296
  12   6   2   2   20981337
  12   7   1   2   1585584
  12   1   8   1   88803
  12   2   7   1   1585584
  12   3   6   1   8654646
  12   4   5   1   19324305
  12   5   4   1   19324305
  12   6   3   1   8654646
  12   7   2   1   1585584
  12   8   1   1   88803

  12         sum   1805010948

  13   1   1   9   183183
  13   1   2   8   4114110
  13   2   1   8   4114110
  13   1   3   7   29135106
  13   2   2   7   70367479
  13   3   1   7   29135106
  13   1   4   6   87933846
  13   2   3   6   374127663
  13   3   2   6   374127663
  13   4   1   6   87933846
  13   1   5   5   125855730
  13   2   4   5   824962502
  13   3   3   5   1453414846
  13   4   2   5   824962502
  13   5   1   5   125855730
  13   1   6   4   87933846
  13   2   5   4   824962502
  13   3   4   4   2239280420
  13   4   3   4   2239280420
  13   5   2   4   824962502
  13   6   1   4   87933846
  13   1   7   3   29135106
  13   2   6   3   374127663
  13   3   5   3   1453414846
  13   4   4   3   2239280420
  13   5   3   3   1453414846
  13   6   2   3   374127663
  13   7   1   3   29135106
  13   1   8   2   4114110
  13   2   7   2   70367479
  13   3   6   2   374127663
  13   4   5   2   824962502
  13   5   4   2   824962502
  13   6   3   2   374127663
  13   7   2   2   70367479
  13   8   1   2   4114110
  13   1   9   1   183183
  13   2   8   1   4114110
  13   3   7   1   29135106
  13   4   6   1   87933846
  13   5   5   1   125855730
  13   6   4   1   87933846
  13   7   3   1   29135106
  13   8   2   1   4114110
  13   9   1   1   183183

  13         sum   19588944336

  14   1   1  10   355355
  14   1   2   9   9798789
  14   2   1   9   9798789
  14   1   3   8   87291204
  14   2   2   8   210164227
  14   3   1   8   87291204
  14   1   4   7   341825484
  14   2   3   7   1444432612
  14   3   2   7   1444432612
  14   4   1   7   341825484
  14   1   5   6   661320660
  14   2   4   6   4286172247
  14   3   3   6   7523770016
  14   4   2   6   4286172247
  14   5   1   6   661320660
  14   1   6   5   661320660
  14   2   5   5   6100939726
  14   3   4   5   16427471172
  14   4   3   5   16427471172
  14   5   2   5   6100939726
  14   6   1   5   661320660
  14   1   7   4   341825484
  14   2   6   4   4286172247
  14   3   5   4   16427471172
  14   4   4   4   25199010256
  14   5   3   4   16427471172
  14   6   2   4   4286172247
  14   7   1   4   341825484
  14   1   8   3   87291204
  14   2   7   3   1444432612
  14   3   6   3   7523770016
  14   4   5   3   16427471172
  14   5   4   3   16427471172
  14   6   3   3   7523770016
  14   7   2   3   1444432612
  14   8   1   3   87291204
  14   1   9   2   9798789
  14   2   8   2   210164227
  14   3   7   2   1444432612
  14   4   6   2   4286172247
  14   5   5   2   6100939726
  14   6   4   2   4286172247
  14   7   3   2   1444432612
  14   8   2   2   210164227
  14   9   1   2   9798789
  14   1  10   1   355355
  14   2   9   1   9798789
  14   3   8   1   87291204
  14   4   7   1   341825484
  14   5   6   1   661320660
  14   6   5   1   661320660
  14   7   4   1   341825484
  14   8   3   1   87291204
  14   9   2   1   9798789
  14  10   1   1   355355

  14         sum   206254571236

\end{verbatim}
\end{scriptsize}
\end{multicols}

\subsection{Genus 3}
\begin{multicols}{3}
\begin{scriptsize}
\begin{verbatim}
   d   v   e   f   h
   7   1   1   1   180

   7         sum   180

   8   1   1   2   3044
   8   1   2   1   3044
   8   2   1   1   3044

   8         sum   9132

   9   1   1   3   26060
   9   1   2   2   63600
   9   2   1   2   63600
   9   1   3   1   26060
   9   2   2   1   63600
   9   3   1   1   26060

   9         sum   268980

  10   1   1   4   152900
  10   1   2   3   659340
  10   2   1   3   659340
  10   1   3   2   659340
  10   2   2   2   1595480
  10   3   1   2   659340
  10   1   4   1   152900
  10   2   3   1   659340
  10   3   2   1   659340
  10   4   1   1   152900

  10         sum   6010220

  11   1   1   5   696905
  11   1   2   4   4606910
  11   2   1   4   4606910
  11   1   3   3   8141100
  11   2   2   3   19571123
  11   3   1   3   8141100
  11   1   4   2   4606910
  11   2   3   2   19571123
  11   3   2   2   19571123
  11   4   1   2   4606910
  11   1   5   1   696905
  11   2   4   1   4606910
  11   3   3   1   8141100
  11   4   2   1   4606910
  11   5   1   1   696905

  11         sum   112868844

  12   1   1   6   2641925
  12   1   2   5   24656775
  12   2   1   5   24656775
  12   1   3   4   66805310
  12   2   2   4   159762815
  12   3   1   4   66805310
  12   1   4   3   66805310
  12   2   3   3   280514670
  12   3   2   3   280514670
  12   4   1   3   66805310
  12   1   5   2   24656775
  12   2   4   2   159762815
  12   3   3   2   280514670
  12   4   2   2   159762815
  12   5   1   2   24656775
  12   1   6   1   2641925
  12   2   5   1   24656775
  12   3   4   1   66805310
  12   4   3   1   66805310
  12   5   2   1   24656775
  12   6   1   1   2641925

  12         sum   1877530740

  13   1   1   7   8691683
  13   1   2   6   108452916
  13   2   1   6   108452916
  13   1   3   5   414918075
  13   2   2   5   988043771
  13   3   1   5   414918075
  13   1   4   4   636184120
  13   2   3   4   2646424729
  13   3   2   4   2646424729
  13   4   1   4   636184120
  13   1   5   3   414918075
  13   2   4   3   2646424729
  13   3   3   3   4623070842
  13   4   2   3   2646424729
  13   5   1   3   414918075
  13   1   6   2   108452916
  13   2   5   2   988043771
  13   3   4   2   2646424729
  13   4   3   2   2646424729
  13   5   2   2   988043771
  13   6   1   2   108452916
  13   1   7   1   8691683
  13   2   6   1   108452916
  13   3   5   1   414918075
  13   4   4   1   636184120
  13   5   3   1   414918075
  13   6   2   1   108452916
  13   7   1   1   8691683

  13         sum   28540603884

  14   1   1   8   25537655
  14   1   2   7   409732895
  14   2   1   7   409732895
  14   1   3   6   2096068975
  14   2   2   6   4973691275
  14   3   1   6   2096068975
  14   1   4   5   4538348815
  14   2   3   5   18733893115
  14   3   2   5   18733893115
  14   4   1   5   4538348815
  14   1   5   4   4538348815
  14   2   4   4   28579309570
  14   3   3   4   49719495672
  14   4   2   4   28579309570
  14   5   1   4   4538348815
  14   1   6   3   2096068975
  14   2   5   3   18733893115
  14   3   4   3   49719495672
  14   4   3   3   49719495672
  14   5   2   3   18733893115
  14   6   1   3   2096068975
  14   1   7   2   409732895
  14   2   6   2   4973691275
  14   3   5   2   18733893115
  14   4   4   2   28579309570
  14   5   3   2   18733893115
  14   6   2   2   4973691275
  14   7   1   2   409732895
  14   1   8   1   25537655
  14   2   7   1   409732895
  14   3   6   1   2096068975
  14   4   5   1   4538348815
  14   5   4   1   4538348815
  14   6   3   1   2096068975
  14   7   2   1   409732895
  14   8   1   1   25537655

  14         sum   404562365316

\end{verbatim}
\end{scriptsize}
\end{multicols}

\subsection{Genus 4}
\begin{multicols}{3}
\begin{scriptsize}
\begin{verbatim}
   d   v   e   f   h
   9   1   1   1   8064

   9         sum   8064

  10   1   1   2   193248
  10   1   2   1   193248
  10   2   1   1   193248

  10         sum   579744

  11   1   1   3   2286636
  11   1   2   2   5458464
  11   2   1   2   5458464
  11   1   3   1   2286636
  11   2   2   1   5458464
  11   3   1   1   2286636

  11         sum   23235300

  12   1   1   4   18128396
  12   1   2   3   75220860
  12   2   1   3   75220860
  12   1   3   2   75220860
  12   2   2   2   178462816
  12   3   1   2   75220860
  12   1   4   1   18128396
  12   2   3   1   75220860
  12   3   2   1   75220860
  12   4   1   1   18128396

  12         sum   684173164

  13   1   1   5   109425316
  13   1   2   4   687238552
  13   2   1   4   687238552
  13   1   3   3   1194737544
  13   2   2   3   2820651496
  13   3   1   3   1194737544
  13   1   4   2   687238552
  13   2   3   2   2820651496
  13   3   2   2   2820651496
  13   4   1   2   687238552
  13   1   5   1   109425316
  13   2   4   1   687238552
  13   3   3   1   1194737544
  13   4   2   1   687238552
  13   5   1   1   109425316

  13         sum   16497874380

  14   1   1   6   539651112
  14   1   2   5   4736419688
  14   2   1   5   4736419688
  14   1   3   4   12465308856
  14   2   2   4   29310854804
  14   3   1   4   12465308856
  14   1   4   3   12465308856
  14   2   3   3   50713072144
  14   3   2   3   50713072144
  14   4   1   3   12465308856
  14   1   5   2   4736419688
  14   2   4   2   29310854804
  14   3   3   2   50713072144
  14   4   2   2   29310854804
  14   5   1   2   4736419688
  14   1   6   1   539651112
  14   2   5   1   4736419688
  14   3   4   1   12465308856
  14   4   3   1   12465308856
  14   5   2   1   4736419688
  14   6   1   1   539651112

  14         sum   344901105444

\end{verbatim}
\end{scriptsize}
\end{multicols}

\subsection{Genus 5}
\begin{multicols}{3}
\begin{scriptsize}
\begin{verbatim}
   d   v   e   f   h
  11   1   1   1   604800

  11         sum   604800

  12   1   1   2   19056960
  12   1   2   1   19056960
  12   2   1   1   19056960

  12         sum   57170880

  13   1   1   3   292271616
  13   1   2   2   686597184
  13   2   1   2   686597184
  13   1   3   1   292271616
  13   2   2   1   686597184
  13   3   1   1   292271616

  13         sum   2936606400

  14   1   1   4   2961802480
  14   1   2   3   11947069680
  14   2   1   3   11947069680
  14   1   3   2   11947069680
  14   2   2   2   27934773440
  14   3   1   2   11947069680
  14   1   4   1   2961802480
  14   2   3   1   11947069680
  14   3   2   1   11947069680
  14   4   1   1   2961802480

  14         sum   108502598960

\end{verbatim}
\end{scriptsize}
\end{multicols}

\subsection{Genus 6}
\begin{multicols}{3}
\begin{scriptsize}
\begin{verbatim}
   d   v   e   f   h
  13   1   1   1   68428800

  13         sum   68428800

  14   1   1   2   2699672832
  14   1   2   1   2699672832
  14   2   1   1   2699672832

  14         sum   8099018496

\end{verbatim}
\end{scriptsize}
\end{multicols}

These tables extend to 14 darts the part of Appendix B of~\cite{Walsh15} about
rooted hypermaps. 

\newpage
\section{First numbers of unrooted hypermaps}
\label{unrooted:appendix}

The following sections show the numbers \texttt{H} of unrooted hypermaps of
genus $g$ with $d$ darts, $v$ vertices, $e$ edges and $d - v - e + 2(1 – g)$ faces, 
for $g \leq 6$ and $d \leq 14$.

\subsection{Genus 0}
\begin{multicols}{3}
\begin{scriptsize}
\begin{verbatim}
   d   v   e   f   H
   1   1   1   1   1

   1         sum   1

   2   1   1   2   1
   2   1   2   1   1
   2   2   1   1   1

   2         sum   3

   3   1   1   3   1
   3   1   2   2   1
   3   2   1   2   1
   3   1   3   1   1
   3   2   2   1   1
   3   3   1   1   1

   3         sum   6

   4   1   1   4   1
   4   1   2   3   2
   4   2   1   3   2
   4   1   3   2   2
   4   2   2   2   5
   4   3   1   2   2
   4   1   4   1   1
   4   2   3   1   2
   4   3   2   1   2
   4   4   1   1   1

   4         sum   20

   5   1   1   5   1
   5   1   2   4   2
   5   2   1   4   2
   5   1   3   3   4
   5   2   2   3   11
   5   3   1   3   4
   5   1   4   2   2
   5   2   3   2   11
   5   3   2   2   11
   5   4   1   2   2
   5   1   5   1   1
   5   2   4   1   2
   5   3   3   1   4
   5   4   2   1   2
   5   5   1   1   1

   5         sum   60

   6   1   1   6   1
   6   1   2   5   3
   6   2   1   5   3
   6   1   3   4   10
   6   2   2   4   24
   6   3   1   4   10
   6   1   4   3   10
   6   2   3   3   46
   6   3   2   3   46
   6   4   1   3   10
   6   1   5   2   3
   6   2   4   2   24
   6   3   3   2   46
   6   4   2   2   24
   6   5   1   2   3
   6   1   6   1   1
   6   2   5   1   3
   6   3   4   1   10
   6   4   3   1   10
   6   5   2   1   3
   6   6   1   1   1

   6         sum   291

   7   1   1   7   1
   7   1   2   6   3
   7   2   1   6   3
   7   1   3   5   15
   7   2   2   5   40
   7   3   1   5   15
   7   1   4   4   25
   7   2   3   4   127
   7   3   2   4   127
   7   4   1   4   25
   7   1   5   3   15
   7   2   4   3   127
   7   3   3   3   242
   7   4   2   3   127
   7   5   1   3   15
   7   1   6   2   3
   7   2   5   2   40
   7   3   4   2   127
   7   4   3   2   127
   7   5   2   2   40
   7   6   1   2   3
   7   1   7   1   1
   7   2   6   1   3
   7   3   5   1   15
   7   4   4   1   25
   7   5   3   1   15
   7   6   2   1   3
   7   7   1   1   1

   7         sum   1310

   8   1   1   8   1
   8   1   2   7   4
   8   2   1   7   4
   8   1   3   6   26
   8   2   2   6   67
   8   3   1   6   26
   8   1   4   5   64
   8   2   3   5   309
   8   3   2   5   309
   8   4   1   5   64
   8   1   5   4   64
   8   2   4   4   505
   8   3   3   4   946
   8   4   2   4   505
   8   5   1   4   64
   8   1   6   3   26
   8   2   5   3   309
   8   3   4   3   946
   8   4   3   3   946
   8   5   2   3   309
   8   6   1   3   26
   8   1   7   2   4
   8   2   6   2   67
   8   3   5   2   309
   8   4   4   2   505
   8   5   3   2   309
   8   6   2   2   67
   8   7   1   2   4
   8   1   8   1   1
   8   2   7   1   4
   8   3   6   1   26
   8   4   5   1   64
   8   5   4   1   64
   8   6   3   1   26
   8   7   2   1   4
   8   8   1   1   1

   8         sum   6975

   9   1   1   9   1
   9   1   2   8   4
   9   2   1   8   4
   9   1   3   7   38
   9   2   2   7   98
   9   3   1   7   38
   9   1   4   6   132
   9   2   3   6   640
   9   3   2   6   640
   9   4   1   6   132
   9   1   5   5   196
   9   2   4   5   1549
   9   3   3   5   2890
   9   4   2   5   1549
   9   5   1   5   196
   9   1   6   4   132
   9   2   5   4   1549
   9   3   4   4   4671
   9   4   3   4   4671
   9   5   2   4   1549
   9   6   1   4   132
   9   1   7   3   38
   9   2   6   3   640
   9   3   5   3   2890
   9   4   4   3   4671
   9   5   3   3   2890
   9   6   2   3   640
   9   7   1   3   38
   9   1   8   2   4
   9   2   7   2   98
   9   3   6   2   640
   9   4   5   2   1549
   9   5   4   2   1549
   9   6   3   2   640
   9   7   2   2   98
   9   8   1   2   4
   9   1   9   1   1
   9   2   8   1   4
   9   3   7   1   38
   9   4   6   1   132
   9   5   5   1   196
   9   6   4   1   132
   9   7   3   1   38
   9   8   2   1   4
   9   9   1   1   1

   9         sum   37746

  10   1   1  10   1
  10   1   2   9   5
  10   2   1   9   5
  10   1   3   8   56
  10   2   2   8   144
  10   3   1   8   56
  10   1   4   7   256
  10   2   3   7   1226
  10   3   2   7   1226
  10   4   1   7   256
  10   1   5   6   536
  10   2   4   6   4111
  10   3   3   6   7606
  10   4   2   6   4111
  10   5   1   6   536
  10   1   6   5   536
  10   2   5   5   6081
  10   3   4   5   18019
  10   4   3   5   18019
  10   5   2   5   6081
  10   6   1   5   536
  10   1   7   4   256
  10   2   6   4   4111
  10   3   5   4   18019
  10   4   4   4   28852
  10   5   3   4   18019
  10   6   2   4   4111
  10   7   1   4   256
  10   1   8   3   56
  10   2   7   3   1226
  10   3   6   3   7606
  10   4   5   3   18019
  10   5   4   3   18019
  10   6   3   3   7606
  10   7   2   3   1226
  10   8   1   3   56
  10   1   9   2   5
  10   2   8   2   144
  10   3   7   2   1226
  10   4   6   2   4111
  10   5   5   2   6081
  10   6   4   2   4111
  10   7   3   2   1226
  10   8   2   2   144
  10   9   1   2   5
  10   1  10   1   1
  10   2   9   1   5
  10   3   8   1   56
  10   4   7   1   256
  10   5   6   1   536
  10   6   5   1   536
  10   7   4   1   256
  10   8   3   1   56
  10   9   2   1   5
  10  10   1   1   1

  10         sum   215602

  11   1   1  11   1
  11   1   2  10   5
  11   2   1  10   5
  11   1   3   9   75
  11   2   2   9   195
  11   3   1   9   75
  11   1   4   8   450
  11   2   3   8   2154
  11   3   2   8   2154
  11   4   1   8   450
  11   1   5   7   1260
  11   2   4   7   9585
  11   3   3   7   17664
  11   4   2   7   9585
  11   5   1   7   1260
  11   1   6   6   1764
  11   2   5   6   19691
  11   3   4   6   57715
  11   4   3   6   57715
  11   5   2   6   19691
  11   6   1   6   1764
  11   1   7   5   1260
  11   2   6   5   19691
  11   3   5   5   84714
  11   4   4   5   134750
  11   5   3   5   84714
  11   6   2   5   19691
  11   7   1   5   1260
  11   1   8   4   450
  11   2   7   4   9585
  11   3   6   4   57715
  11   4   5   4   134750
  11   5   4   4   134750
  11   6   3   4   57715
  11   7   2   4   9585
  11   8   1   4   450
  11   1   9   3   75
  11   2   8   3   2154
  11   3   7   3   17664
  11   4   6   3   57715
  11   5   5   3   84714
  11   6   4   3   57715
  11   7   3   3   17664
  11   8   2   3   2154
  11   9   1   3   75
  11   1  10   2   5
  11   2   9   2   195
  11   3   8   2   2154
  11   4   7   2   9585
  11   5   6   2   19691
  11   6   5   2   19691
  11   7   4   2   9585
  11   8   3   2   2154
  11   9   2   2   195
  11  10   1   2   5
  11   1  11   1   1
  11   2  10   1   5
  11   3   9   1   75
  11   4   8   1   450
  11   5   7   1   1260
  11   6   6   1   1764
  11   7   5   1   1260
  11   8   4   1   450
  11   9   3   1   75
  11  10   2   1   5
  11  11   1   1   1

  11         sum   1262874

  12   1   1  12   1
  12   1   2  11   6
  12   2   1  11   6
  12   1   3  10   104
  12   2   2  10   265
  12   3   1  10   104
  12   1   4   9   765
  12   2   3   9   3605
  12   3   2   9   3605
  12   4   1   9   765
  12   1   5   8   2736
  12   2   4   8   20472
  12   3   3   8   37545
  12   4   2   8   20472
  12   5   1   8   2736
  12   1   6   7   5102
  12   2   5   7   55633
  12   3   4   7   161455
  12   4   3   7   161455
  12   5   2   7   55633
  12   6   1   7   5102
  12   1   7   6   5102
  12   2   6   6   77174
  12   3   5   6   326432
  12   4   4   6   516507
  12   5   3   6   326432
  12   6   2   6   77174
  12   7   1   6   5102
  12   1   8   5   2736
  12   2   7   5   55633
  12   3   6   5   326432
  12   4   5   5   752940
  12   5   4   5   752940
  12   6   3   5   326432
  12   7   2   5   55633
  12   8   1   5   2736
  12   1   9   4   765
  12   2   8   4   20472
  12   3   7   4   161455
  12   4   6   4   516507
  12   5   5   4   752940
  12   6   4   4   516507
  12   7   3   4   161455
  12   8   2   4   20472
  12   9   1   4   765
  12   1  10   3   104
  12   2   9   3   3605
  12   3   8   3   37545
  12   4   7   3   161455
  12   5   6   3   326432
  12   6   5   3   326432
  12   7   4   3   161455
  12   8   3   3   37545
  12   9   2   3   3605
  12  10   1   3   104
  12   1  11   2   6
  12   2  10   2   265
  12   3   9   2   3605
  12   4   8   2   20472
  12   5   7   2   55633
  12   6   6   2   77174
  12   7   5   2   55633
  12   8   4   2   20472
  12   9   3   2   3605
  12  10   2   2   265
  12  11   1   2   6
  12   1  12   1   1
  12   2  11   1   6
  12   3  10   1   104
  12   4   9   1   765
  12   5   8   1   2736
  12   6   7   1   5102
  12   7   6   1   5102
  12   8   5   1   2736
  12   9   4   1   765
  12  10   3   1   104
  12  11   2   1   6
  12  12   1   1   1

  12         sum   7611156

  13   1   1  13   1
  13   1   2  12   6
  13   2   1  12   6
  13   1   3  11   132
  13   2   2  11   341
  13   3   1  11   132
  13   1   4  10   1210
  13   2   3  10   5709
  13   3   2  10   5709
  13   4   1  10   1210
  13   1   5   9   5445
  13   2   4   9   40425
  13   3   3   9   73920
  13   4   2   9   40425
  13   5   1   9   5445
  13   1   6   8   13068
  13   2   5   8   140591
  13   3   4   8   404965
  13   4   3   8   404965
  13   5   2   8   140591
  13   6   1   8   13068
  13   1   7   7   17424
  13   2   6   7   258194
  13   3   5   7   1078456
  13   4   4   7   1699200
  13   5   3   7   1078456
  13   6   2   7   258194
  13   7   1   7   17424
  13   1   8   6   13068
  13   2   7   6   258194
  13   3   6   6   1485778
  13   4   5   6   3395140
  13   5   4   6   3395140
  13   6   3   6   1485778
  13   7   2   6   258194
  13   8   1   6   13068
  13   1   9   5   5445
  13   2   8   5   140591
  13   3   7   5   1078456
  13   4   6   5   3395140
  13   5   5   5   4924094
  13   6   4   5   3395140
  13   7   3   5   1078456
  13   8   2   5   140591
  13   9   1   5   5445
  13   1  10   4   1210
  13   2   9   4   40425
  13   3   8   4   404965
  13   4   7   4   1699200
  13   5   6   4   3395140
  13   6   5   4   3395140
  13   7   4   4   1699200
  13   8   3   4   404965
  13   9   2   4   40425
  13  10   1   4   1210
  13   1  11   3   132
  13   2  10   3   5709
  13   3   9   3   73920
  13   4   8   3   404965
  13   5   7   3   1078456
  13   6   6   3   1485778
  13   7   5   3   1078456
  13   8   4   3   404965
  13   9   3   3   73920
  13  10   2   3   5709
  13  11   1   3   132
  13   1  12   2   6
  13   2  11   2   341
  13   3  10   2   5709
  13   4   9   2   40425
  13   5   8   2   140591
  13   6   7   2   258194
  13   7   6   2   258194
  13   8   5   2   140591
  13   9   4   2   40425
  13  10   3   2   5709
  13  11   2   2   341
  13  12   1   2   6
  13   1  13   1   1
  13   2  12   1   6
  13   3  11   1   132
  13   4  10   1   1210
  13   5   9   1   5445
  13   6   8   1   13068
  13   7   7   1   17424
  13   8   6   1   13068
  13   9   5   1   5445
  13  10   4   1   1210
  13  11   3   1   132
  13  12   2   1   6
  13  13   1   1   1

  13         sum   46814132

  14   1   1  14   1
  14   1   2  13   7
  14   2   1  13   7
  14   1   3  12   172
  14   2   2  12   440
  14   3   1  12   172
  14   1   4  11   1868
  14   2   3  11   8741
  14   3   2  11   8741
  14   4   1  11   1868
  14   1   5  10   10247
  14   2   4  10   75283
  14   3   3  10   137217
  14   4   2  10   75283
  14   5   1  10   10247
  14   1   6   9   30711
  14   2   5   9   325652
  14   3   4   9   931845
  14   4   3   9   931845
  14   5   2   9   325652
  14   6   1   9   30711
  14   1   7   8   52634
  14   2   6   8   764633
  14   3   5   8   3159069
  14   4   4   8   4960016
  14   5   3   8   3159069
  14   6   2   8   764633
  14   7   1   8   52634
  14   1   8   7   52634
  14   2   7   7   1012271
  14   3   6   7   5731330
  14   4   5   7   12995424
  14   5   4   7   12995424
  14   6   3   7   5731330
  14   7   2   7   1012271
  14   8   1   7   52634
  14   1   9   6   30711
  14   2   8   6   764633
  14   3   7   6   5731330
  14   4   6   6   17809776
  14   5   5   6   25720986
  14   6   4   6   17809776
  14   7   3   6   5731330
  14   8   2   6   764633
  14   9   1   6   30711
  14   1  10   5   10247
  14   2   9   5   325652
  14   3   8   5   3159069
  14   4   7   5   12995424
  14   5   6   5   25720986
  14   6   5   5   25720986
  14   7   4   5   12995424
  14   8   3   5   3159069
  14   9   2   5   325652
  14  10   1   5   10247
  14   1  11   4   1868
  14   2  10   4   75283
  14   3   9   4   931845
  14   4   8   4   4960016
  14   5   7   4   12995424
  14   6   6   4   17809776
  14   7   5   4   12995424
  14   8   4   4   4960016
  14   9   3   4   931845
  14  10   2   4   75283
  14  11   1   4   1868
  14   1  12   3   172
  14   2  11   3   8741
  14   3  10   3   137217
  14   4   9   3   931845
  14   5   8   3   3159069
  14   6   7   3   5731330
  14   7   6   3   5731330
  14   8   5   3   3159069
  14   9   4   3   931845
  14  10   3   3   137217
  14  11   2   3   8741
  14  12   1   3   172
  14   1  13   2   7
  14   2  12   2   440
  14   3  11   2   8741
  14   4  10   2   75283
  14   5   9   2   325652
  14   6   8   2   764633
  14   7   7   2   1012271
  14   8   6   2   764633
  14   9   5   2   325652
  14  10   4   2   75283
  14  11   3   2   8741
  14  12   2   2   440
  14  13   1   2   7
  14   1  14   1   1
  14   2  13   1   7
  14   3  12   1   172
  14   4  11   1   1868
  14   5  10   1   10247
  14   6   9   1   30711
  14   7   8   1   52634
  14   8   7   1   52634
  14   9   6   1   30711
  14  10   5   1   10247
  14  11   4   1   1868
  14  12   3   1   172
  14  13   2   1   7
  14  14   1   1   1

  14         sum   293447817
\end{verbatim}
\end{scriptsize}
\end{multicols}

\newpage
\subsection{Genus 1}
\begin{multicols}{3}
\begin{scriptsize}
\begin{verbatim}
   d   v   e   f   H
   3   1   1   1   1

   3         sum   1

   4   1   1   2   2
   4   1   2   1   2
   4   2   1   1   2

   4         sum   6

   5   1   1   3   3
   5   1   2   2   8
   5   2   1   2   8
   5   1   3   1   3
   5   2   2   1   8
   5   3   1   1   3

   5         sum   33

   6   1   1   4   7
   6   1   2   3   31
   6   2   1   3   31
   6   1   3   2   31
   6   2   2   2   78
   6   3   1   2   31
   6   1   4   1   7
   6   2   3   1   31
   6   3   2   1   31
   6   4   1   1   7

   6         sum   285

   7   1   1   5   10
   7   1   2   4   80
   7   2   1   4   80
   7   1   3   3   150
   7   2   2   3   385
   7   3   1   3   150
   7   1   4   2   80
   7   2   3   2   385
   7   3   2   2   385
   7   4   1   2   80
   7   1   5   1   10
   7   2   4   1   80
   7   3   3   1   150
   7   4   2   1   80
   7   5   1   1   10

   7         sum   2115

   8   1   1   6   17
   8   1   2   5   187
   8   2   1   5   187
   8   1   3   4   557
   8   2   2   4   1409
   8   3   1   4   557
   8   1   4   3   557
   8   2   3   3   2597
   8   3   2   3   2597
   8   4   1   3   557
   8   1   5   2   187
   8   2   4   2   1409
   8   3   3   2   2597
   8   4   2   2   1409
   8   5   1   2   187
   8   1   6   1   17
   8   2   5   1   187
   8   3   4   1   557
   8   4   3   1   557
   8   5   2   1   187
   8   6   1   1   17

   8         sum   16533

   9   1   1   7   24
   9   1   2   6   374
   9   2   1   6   374
   9   1   3   5   1634
   9   2   2   5   4115
   9   3   1   5   1634
   9   1   4   4   2616
   9   2   3   4   12033
   9   3   2   4   12033
   9   4   1   4   2616
   9   1   5   3   1634
   9   2   4   3   12033
   9   3   3   3   21990
   9   4   2   3   12033
   9   5   1   3   1634
   9   1   6   2   374
   9   2   5   2   4115
   9   3   4   2   12033
   9   4   3   2   12033
   9   5   2   2   4115
   9   6   1   2   374
   9   1   7   1   24
   9   2   6   1   374
   9   3   5   1   1634
   9   4   4   1   2616
   9   5   3   1   1634
   9   6   2   1   374
   9   7   1   1   24

   9         sum   126501

  10   1   1   8   34
  10   1   2   7   698
  10   2   1   7   698
  10   1   3   6   4172
  10   2   2   6   10434
  10   3   1   6   4172
  10   1   4   5   9724
  10   2   3   5   44091
  10   3   2   5   44091
  10   4   1   5   9724
  10   1   5   4   9724
  10   2   4   4   69790
  10   3   3   4   126519
  10   4   2   4   69790
  10   5   1   4   9724
  10   1   6   3   4172
  10   2   5   3   44091
  10   3   4   3   126519
  10   4   3   3   126519
  10   5   2   3   44091
  10   6   1   3   4172
  10   1   7   2   698
  10   2   6   2   10434
  10   3   5   2   44091
  10   4   4   2   69790
  10   5   3   2   44091
  10   6   2   2   10434
  10   7   1   2   698
  10   1   8   1   34
  10   2   7   1   698
  10   3   6   1   4172
  10   4   5   1   9724
  10   5   4   1   9724
  10   6   3   1   4172
  10   7   2   1   698
  10   8   1   1   34

  10         sum   972441

  11   1   1   9   45
  11   1   2   8   1200
  11   2   1   8   1200
  11   1   3   7   9450
  11   2   2   7   23547
  11   3   1   7   9450
  11   1   4   6   30240
  11   2   3   6   135775
  11   3   2   6   135775
  11   4   1   6   30240
  11   1   5   5   44100
  11   2   4   5   310985
  11   3   3   5   560498
  11   4   2   5   310985
  11   5   1   5   44100
  11   1   6   4   30240
  11   2   5   4   310985
  11   3   4   4   880403
  11   4   3   4   880403
  11   5   2   4   310985
  11   6   1   4   30240
  11   1   7   3   9450
  11   2   6   3   135775
  11   3   5   3   560498
  11   4   4   3   880403
  11   5   3   3   560498
  11   6   2   3   135775
  11   7   1   3   9450
  11   1   8   2   1200
  11   2   7   2   23547
  11   3   6   2   135775
  11   4   5   2   310985
  11   5   4   2   310985
  11   6   3   2   135775
  11   7   2   2   23547
  11   8   1   2   1200
  11   1   9   1   45
  11   2   8   1   1200
  11   3   7   1   9450
  11   4   6   1   30240
  11   5   5   1   44100
  11   6   4   1   30240
  11   7   3   1   9450
  11   8   2   1   1200
  11   9   1   1   45

  11         sum   7451679

  12   1   1  10   62
  12   1   2   9   1976
  12   2   1   9   1976
  12   1   3   8   19694
  12   2   2   8   48846
  12   3   1   8   19694
  12   1   4   7   82652
  12   2   3   7   367645
  12   3   2   7   367645
  12   4   1   7   82652
  12   1   5   6   165262
  12   2   4   6   1147628
  12   3   3   6   2058329
  12   4   2   6   1147628
  12   5   1   6   165262
  12   1   6   5   165262
  12   2   5   5   1660331
  12   3   4   5   4649379
  12   4   3   5   4649379
  12   5   2   5   1660331
  12   6   1   5   165262
  12   1   7   4   82652
  12   2   6   4   1147628
  12   3   5   4   4649379
  12   4   4   4   7259140
  12   5   3   4   4649379
  12   6   2   4   1147628
  12   7   1   4   82652
  12   1   8   3   19694
  12   2   7   3   367645
  12   3   6   3   2058329
  12   4   5   3   4649379
  12   5   4   3   4649379
  12   6   3   3   2058329
  12   7   2   3   367645
  12   8   1   3   19694
  12   1   9   2   1976
  12   2   8   2   48846
  12   3   7   2   367645
  12   4   6   2   1147628
  12   5   5   2   1660331
  12   6   4   2   1147628
  12   7   3   2   367645
  12   8   2   2   48846
  12   9   1   2   1976
  12   1  10   1   62
  12   2   9   1   1976
  12   3   8   1   19694
  12   4   7   1   82652
  12   5   6   1   165262
  12   6   5   1   165262
  12   7   4   1   82652
  12   8   3   1   19694
  12   9   2   1   1976
  12  10   1   1   62

  12         sum   57167260

  13   1   1  11   77
  13   1   2  10   3080
  13   2   1  10   3080
  13   1   3   9   38115
  13   2   2   9   94281
  13   3   1   9   38115
  13   1   4   8   203280
  13   2   3   8   898051
  13   3   2   8   898051
  13   4   1   8   203280
  13   1   5   7   533610
  13   2   4   7   3661896
  13   3   3   7   6542368
  13   4   2   7   3661896
  13   5   1   7   533610
  13   1   6   6   731808
  13   2   5   6   7221592
  13   3   4   6   20047636
  13   4   3   6   20047636
  13   5   2   6   7221592
  13   6   1   6   731808
  13   1   7   5   533610
  13   2   6   5   7221592
  13   3   5   5   28831218
  13   4   4   5   44800675
  13   5   3   5   28831218
  13   6   2   5   7221592
  13   7   1   5   533610
  13   1   8   4   203280
  13   2   7   4   3661896
  13   3   6   4   20047636
  13   4   5   4   44800675
  13   5   4   4   44800675
  13   6   3   4   20047636
  13   7   2   4   3661896
  13   8   1   4   203280
  13   1   9   3   38115
  13   2   8   3   898051
  13   3   7   3   6542368
  13   4   6   3   20047636
  13   5   5   3   28831218
  13   6   4   3   20047636
  13   7   3   3   6542368
  13   8   2   3   898051
  13   9   1   3   38115
  13   1  10   2   3080
  13   2   9   2   94281
  13   3   8   2   898051
  13   4   7   2   3661896
  13   5   6   2   7221592
  13   6   5   2   7221592
  13   7   4   2   3661896
  13   8   3   2   898051
  13   9   2   2   94281
  13  10   1   2   3080
  13   1  11   1   77
  13   2  10   1   3080
  13   3   9   1   38115
  13   4   8   1   203280
  13   5   7   1   533610
  13   6   6   1   731808
  13   7   5   1   533610
  13   8   4   1   203280
  13   9   3   1   38115
  13  10   2   1   3080
  13  11   1   1   77

  13         sum   438644841

  14   1   1  12   99
  14   1   2  11   4659
  14   2   1  11   4659
  14   1   3  10   69765
  14   2   2  10   172040
  14   3   1  10   69765
  14   1   4   9   460245
  14   2   3   9   2020530
  14   3   2   9   2020530
  14   4   1   9   460245
  14   1   5   8   1533950
  14   2   4   8   10419653
  14   3   3   8   18554641
  14   4   2   8   10419653
  14   5   1   8   1533950
  14   1   6   7   2760990
  14   2   5   7   26837442
  14   3   4   7   73967488
  14   4   3   7   73967488
  14   5   2   7   26837442
  14   6   1   7   2760990
  14   1   7   6   2760990
  14   2   6   6   36580432
  14   3   5   6   144298902
  14   4   4   6   223353280
  14   5   3   6   144298902
  14   6   2   6   36580432
  14   7   1   6   2760990
  14   1   8   5   1533950
  14   2   7   5   26837442
  14   3   6   5   144298902
  14   4   5   5   319684549
  14   5   4   5   319684549
  14   6   3   5   144298902
  14   7   2   5   26837442
  14   8   1   5   1533950
  14   1   9   4   460245
  14   2   8   4   10419653
  14   3   7   4   73967488
  14   4   6   4   223353280
  14   5   5   4   319684549
  14   6   4   4   223353280
  14   7   3   4   73967488
  14   8   2   4   10419653
  14   9   1   4   460245
  14   1  10   3   69765
  14   2   9   3   2020530
  14   3   8   3   18554641
  14   4   7   3   73967488
  14   5   6   3   144298902
  14   6   5   3   144298902
  14   7   4   3   73967488
  14   8   3   3   18554641
  14   9   2   3   2020530
  14  10   1   3   69765
  14   1  11   2   4659
  14   2  10   2   172040
  14   3   9   2   2020530
  14   4   8   2   10419653
  14   5   7   2   26837442
  14   6   6   2   36580432
  14   7   5   2   26837442
  14   8   4   2   10419653
  14   9   3   2   2020530
  14  10   2   2   172040
  14  11   1   2   4659
  14   1  12   1   99
  14   2  11   1   4659
  14   3  10   1   69765
  14   4   9   1   460245
  14   5   8   1   1533950
  14   6   7   1   2760990
  14   7   6   1   2760990
  14   8   5   1   1533950
  14   9   4   1   460245
  14  10   3   1   69765
  14  11   2   1   4659
  14  12   1   1   99

  14         sum   3369276867
\end{verbatim}
\end{scriptsize}
\end{multicols}

\newpage
\subsection{Genus 2}
\begin{multicols}{3}
\begin{scriptsize}
\begin{verbatim}
   d   v   e   f   H
   5   1   1   1   4

   5         sum   4

   6   1   1   2   16
   6   1   2   1   16
   6   2   1   1   16

   6         sum   48

   7   1   1   3   67
   7   1   2   2   169
   7   2   1   2   169
   7   1   3   1   67
   7   2   2   1   169
   7   3   1   1   67

   7         sum   708

   8   1   1   4   237
   8   1   2   3   1072
   8   2   1   3   1072
   8   1   3   2   1072
   8   2   2   2   2664
   8   3   1   2   1072
   8   1   4   1   237
   8   2   3   1   1072
   8   3   2   1   1072
   8   4   1   1   237

   8         sum   9807

   9   1   1   5   667
   9   1   2   4   4736
   9   2   1   4   4736
   9   1   3   3   8560
   9   2   2   3   21113
   9   3   1   3   8560
   9   1   4   2   4736
   9   2   3   2   21113
   9   3   2   2   21113
   9   4   1   2   4736
   9   1   5   1   667
   9   2   4   1   4736
   9   3   3   1   8560
   9   4   2   1   4736
   9   5   1   1   667

   9         sum   119436

  10   1   1   6   1649
  10   1   2   5   16725
  10   2   1   5   16725
  10   1   3   4   47164
  10   2   2   4   115478
  10   3   1   4   47164
  10   1   4   3   47164
  10   2   3   3   206895
  10   3   2   3   206895
  10   4   1   3   47164
  10   1   5   2   16725
  10   2   4   2   115478
  10   3   3   2   206895
  10   4   2   2   115478
  10   5   1   2   16725
  10   1   6   1   1649
  10   2   5   1   16725
  10   3   4   1   47164
  10   4   3   1   47164
  10   5   2   1   16725
  10   6   1   1   1649

  10         sum   1355400

  11   1   1   7   3633
  11   1   2   6   50001
  11   2   1   6   50001
  11   1   3   5   201915
  11   2   2   5   491729
  11   3   1   5   201915
  11   1   4   4   315000
  11   2   3   4   1364986
  11   3   2   4   1364986
  11   4   1   4   315000
  11   1   5   3   201915
  11   2   4   3   1364986
  11   3   3   3   2428862
  11   4   2   3   1364986
  11   5   1   3   201915
  11   1   6   2   50001
  11   2   5   2   491729
  11   3   4   2   1364986
  11   4   3   2   1364986
  11   5   2   2   491729
  11   6   1   2   50001
  11   1   7   1   3633
  11   2   6   1   50001
  11   3   5   1   201915
  11   4   4   1   315000
  11   5   3   1   201915
  11   6   2   1   50001
  11   7   1   1   3633

  11         sum   14561360

  12   1   1   8   7417
  12   1   2   7   132202
  12   2   1   7   132202
  12   1   3   6   721382
  12   2   2   6   1748723
  12   3   1   6   721382
  12   1   4   5   1610617
  12   2   3   5   6908644
  12   3   2   5   6908644
  12   4   1   5   1610617
  12   1   5   4   1610617
  12   2   4   4   10702449
  12   3   3   4   18938994
  12   4   2   4   10702449
  12   5   1   4   1610617
  12   1   6   3   721382
  12   2   5   3   6908644
  12   3   4   3   18938994
  12   4   3   3   18938994
  12   5   2   3   6908644
  12   6   1   3   721382
  12   1   7   2   132202
  12   2   6   2   1748723
  12   3   5   2   6908644
  12   4   4   2   10702449
  12   5   3   2   6908644
  12   6   2   2   1748723
  12   7   1   2   132202
  12   1   8   1   7417
  12   2   7   1   132202
  12   3   6   1   721382
  12   4   5   1   1610617
  12   5   4   1   1610617
  12   6   3   1   721382
  12   7   2   1   132202
  12   8   1   1   7417

  12         sum   150429819

  13   1   1   9   14091
  13   1   2   8   316470
  13   2   1   8   316470
  13   1   3   7   2241162
  13   2   2   7   5412883
  13   3   1   7   2241162
  13   1   4   6   6764142
  13   2   3   6   28779051
  13   3   2   6   28779051
  13   4   1   6   6764142
  13   1   5   5   9681210
  13   2   4   5   63458654
  13   3   3   5   111801142
  13   4   2   5   63458654
  13   5   1   5   9681210
  13   1   6   4   6764142
  13   2   5   4   63458654
  13   3   4   4   172252340
  13   4   3   4   172252340
  13   5   2   4   63458654
  13   6   1   4   6764142
  13   1   7   3   2241162
  13   2   6   3   28779051
  13   3   5   3   111801142
  13   4   4   3   172252340
  13   5   3   3   111801142
  13   6   2   3   28779051
  13   7   1   3   2241162
  13   1   8   2   316470
  13   2   7   2   5412883
  13   3   6   2   28779051
  13   4   5   2   63458654
  13   5   4   2   63458654
  13   6   3   2   28779051
  13   7   2   2   5412883
  13   8   1   2   316470
  13   1   9   1   14091
  13   2   8   1   316470
  13   3   7   1   2241162
  13   4   6   1   6764142
  13   5   5   1   9681210
  13   6   4   1   6764142
  13   7   3   1   2241162
  13   8   2   1   316470
  13   9   1   1   14091

  13         sum   1506841872

  14   1   1  10   25405
  14   1   2   9   700045
  14   2   1   9   700045
  14   1   3   8   6235526
  14   2   2   8   15012496
  14   3   1   8   6235526
  14   1   4   7   24417030
  14   2   3   7   103175785
  14   3   2   7   103175785
  14   4   1   7   24417030
  14   1   5   6   47238510
  14   2   4   6   306159286
  14   3   3   6   537417269
  14   4   2   6   306159286
  14   5   1   6   47238510
  14   1   6   5   47238510
  14   2   5   5   435785878
  14   3   4   5   1173398706
  14   4   3   5   1173398706
  14   5   2   5   435785878
  14   6   1   5   47238510
  14   1   7   4   24417030
  14   2   6   4   306159286
  14   3   5   4   1173398706
  14   4   4   4   1799940644
  14   5   3   4   1173398706
  14   6   2   4   306159286
  14   7   1   4   24417030
  14   1   8   3   6235526
  14   2   7   3   103175785
  14   3   6   3   537417269
  14   4   5   3   1173398706
  14   5   4   3   1173398706
  14   6   3   3   537417269
  14   7   2   3   103175785
  14   8   1   3   6235526
  14   1   9   2   700045
  14   2   8   2   15012496
  14   3   7   2   103175785
  14   4   6   2   306159286
  14   5   5   2   435785878
  14   6   4   2   306159286
  14   7   3   2   103175785
  14   8   2   2   15012496
  14   9   1   2   700045
  14   1  10   1   25405
  14   2   9   1   700045
  14   3   8   1   6235526
  14   4   7   1   24417030
  14   5   6   1   47238510
  14   6   5   1   47238510
  14   7   4   1   24417030
  14   8   3   1   6235526
  14   9   2   1   700045
  14  10   1   1   25405

  14         sum   14732613116
\end{verbatim}
\end{scriptsize}
\end{multicols}

\subsection{Genus 3}
\begin{multicols}{3}
\begin{scriptsize}
\begin{verbatim}
   d   v   e   f   H
   7   1   1   1   30

   7         sum   30

   8   1   1   2   385
   8   1   2   1   385
   8   2   1   1   385

   8         sum   1155

   9   1   1   3   2900
   9   1   2   2   7070
   9   2   1   2   7070
   9   1   3   1   2900
   9   2   2   1   7070
   9   3   1   1   2900

   9         sum   29910

  10   1   1   4   15308
  10   1   2   3   65972
  10   2   1   3   65972
  10   1   3   2   65972
  10   2   2   2   159608
  10   3   1   2   65972
  10   1   4   1   15308
  10   2   3   1   65972
  10   3   2   1   65972
  10   4   1   1   15308

  10         sum   601364

  11   1   1   5   63355
  11   1   2   4   418810
  11   2   1   4   418810
  11   1   3   3   740100
  11   2   2   3   1779193
  11   3   1   3   740100
  11   1   4   2   418810
  11   2   3   2   1779193
  11   3   2   2   1779193
  11   4   1   2   418810
  11   1   5   1   63355
  11   2   4   1   418810
  11   3   3   1   740100
  11   4   2   1   418810
  11   5   1   1   63355

  11         sum   10260804

  12   1   1   6   220244
  12   1   2   5   2054974
  12   2   1   5   2054974
  12   1   3   4   5567550
  12   2   2   4   13314231
  12   3   1   4   5567550
  12   1   4   3   5567550
  12   2   3   3   23377106
  12   3   2   3   23377106
  12   4   1   3   5567550
  12   1   5   2   2054974
  12   2   4   2   13314231
  12   3   3   2   23377106
  12   4   2   2   13314231
  12   5   1   2   2054974
  12   1   6   1   220244
  12   2   5   1   2054974
  12   3   4   1   5567550
  12   4   3   1   5567550
  12   5   2   1   2054974
  12   6   1   1   220244

  12         sum   156469887

  13   1   1   7   668591
  13   1   2   6   8342532
  13   2   1   6   8342532
  13   1   3   5   31916775
  13   2   2   5   76003367
  13   3   1   5   31916775
  13   1   4   4   48937240
  13   2   3   4   203571133
  13   3   2   4   203571133
  13   4   1   4   48937240
  13   1   5   3   31916775
  13   2   4   3   203571133
  13   3   3   3   355620834
  13   4   2   3   203571133
  13   5   1   3   31916775
  13   1   6   2   8342532
  13   2   5   2   76003367
  13   3   4   2   203571133
  13   4   3   2   203571133
  13   5   2   2   76003367
  13   6   1   2   8342532
  13   1   7   1   668591
  13   2   6   1   8342532
  13   3   5   1   31916775
  13   4   4   1   48937240
  13   5   3   1   31916775
  13   6   2   1   8342532
  13   7   1   1   668591

  13         sum   2195431068

  14   1   1   8   1824323
  14   1   2   7   29267487
  14   2   1   7   29267487
  14   1   3   6   149721473
  14   2   2   6   355267058
  14   3   1   6   149721473
  14   1   4   5   324171185
  14   2   3   5   1338142324
  14   3   2   5   1338142324
  14   4   1   5   324171185
  14   1   5   4   324171185
  14   2   4   4   2041388556
  14   3   3   4   3551405485
  14   4   2   4   2041388556
  14   5   1   4   324171185
  14   1   6   3   149721473
  14   2   5   3   1338142324
  14   3   4   3   3551405485
  14   4   3   3   3551405485
  14   5   2   3   1338142324
  14   6   1   3   149721473
  14   1   7   2   29267487
  14   2   6   2   355267058
  14   3   5   2   1338142324
  14   4   4   2   2041388556
  14   5   3   2   1338142324
  14   6   2   2   355267058
  14   7   1   2   29267487
  14   1   8   1   1824323
  14   2   7   1   29267487
  14   3   6   1   149721473
  14   4   5   1   324171185
  14   5   4   1   324171185
  14   6   3   1   149721473
  14   7   2   1   29267487
  14   8   1   1   1824323

  14         sum   28897471080
\end{verbatim}
\end{scriptsize}
\end{multicols}

\newpage
\subsection{Genus 4}
\begin{multicols}{3}
\begin{scriptsize}
\begin{verbatim}
   d   v   e   f   H
   9   1   1   1   900

   9         sum   900

  10   1   1   2   19344
  10   1   2   1   19344
  10   2   1   1   19344

  10         sum   58032

  11   1   1   3   207876
  11   1   2   2   496224
  11   2   1   2   496224
  11   1   3   1   207876
  11   2   2   1   496224
  11   3   1   1   207876

  11         sum   2112300

  12   1   1   4   1510846
  12   1   2   3   6268712
  12   2   1   3   6268712
  12   1   3   2   6268712
  12   2   2   2   14872428
  12   3   1   2   6268712
  12   1   4   1   1510846
  12   2   3   1   6268712
  12   3   2   1   6268712
  12   4   1   1   1510846

  12         sum   57017238

  13   1   1   5   8417332
  13   1   2   4   52864504
  13   2   1   4   52864504
  13   1   3   3   91902888
  13   2   2   3   216973192
  13   3   1   3   91902888
  13   1   4   2   52864504
  13   2   3   2   216973192
  13   3   2   2   216973192
  13   4   1   2   52864504
  13   1   5   1   8417332
  13   2   4   1   52864504
  13   3   3   1   91902888
  13   4   2   1   52864504
  13   5   1   1   8417332

  13         sum   1269067260

  14   1   1   6   38547144
  14   1   2   5   338317960
  14   2   1   5   338317960
  14   1   3   4   890383128
  14   2   2   4   2093639428
  14   3   1   4   890383128
  14   1   4   3   890383128
  14   2   3   3   3622371084
  14   3   2   3   3622371084
  14   4   1   3   890383128
  14   1   5   2   338317960
  14   2   4   2   2093639428
  14   3   3   2   3622371084
  14   4   2   2   2093639428
  14   5   1   2   338317960
  14   1   6   1   38547144
  14   2   5   1   338317960
  14   3   4   1   890383128
  14   4   3   1   890383128
  14   5   2   1   338317960
  14   6   1   1   38547144

  14         sum   24635879496
\end{verbatim}
\end{scriptsize}
\end{multicols}

\subsection{Genus 5}
\begin{multicols}{3}
\begin{scriptsize}
\begin{verbatim}
   d   v   e   f   H
  11   1   1   1   54990

  11         sum   54990

  12   1   1   2   1588218
  12   1   2   1   1588218
  12   2   1   1   1588218

  12         sum   4764654

  13   1   1   3   22482432
  13   1   2   2   52815168
  13   2   1   2   52815168
  13   1   3   1   22482432
  13   2   2   1   52815168
  13   3   1   1   22482432

  13         sum   225892800

  14   1   1   4   211558928
  14   1   2   3   853365360
  14   2   1   3   853365360
  14   1   3   2   853365360
  14   2   2   2   1995345826
  14   3   1   2   853365360
  14   1   4   1   211558928
  14   2   3   1   853365360
  14   3   2   1   853365360
  14   4   1   1   211558928

  14         sum   7750214770
\end{verbatim}
\end{scriptsize}
\end{multicols}

\subsection{Genus 6}
\begin{multicols}{3}
\begin{scriptsize}
\begin{verbatim}
   d   v   e   f   H
  13   1   1   1   5263764

  13         sum   5263764

  14   1   1   2   192834612
  14   1   2   1   192834612
  14   2   1   1   192834612

  14         sum   578503836
\end{verbatim}
\end{scriptsize}
\end{multicols}

\end{document}